\documentclass[a4paper,11pt,reqno]{amsart}

\usepackage{amssymb,amsmath,amsfonts} 
\usepackage{a4wide}
\usepackage{color}
\usepackage{amsthm}
\usepackage{enumerate}

\numberwithin{equation}{section}    
\theoremstyle{plain}
\newtheorem{Theorem}{Theorem}[section]
\newtheorem{Proposition}[Theorem]{Proposition}
\newtheorem{Corollary}[Theorem]{Corollary}
\newtheorem{Lemma}[Theorem]{Lemma}
\theoremstyle{definition}

\newtheorem{Definition}[Theorem]{Definition}

\newtheorem{Remark}[Theorem]{Remark}

\renewcommand{\phi}{\varphi}

\newcommand{\ato}[2]{\genfrac{}{}{0pt}{2}{#1}{#2}}

\DeclareMathOperator{\av}{\mathfrak{u}}

\newcommand{\RR}{\mathbb{R}}
\newcommand{\CC}{\mathbb{C}}
\newcommand{\NN}{\mathbb{N}}

\newcommand{\Vfr}{V^{{\rm fin}, r}}

\newcommand{\cA}{\mathcal{A}}

\newcommand{\calB}{\mathcal{B}}

\newcommand{\calF}{\mathcal{F}}
\newcommand{\calG}{\mathcal{G}}

\newcommand{\calK}{\mathcal{K}}
\newcommand{\calM}{\mathcal{M}}
\newcommand{\calN}{\mathcal{N}}

\newcommand{\tr}{{\mathrm{tr}}}

\newcommand{\hm}[1]{\textbf{*}\leavevmode{\marginpar{\tiny%
$\hbox to 0mm{\hspace*{-0.5mm}$\leftarrow$\hss}%
\vcenter{\vrule depth 0.1mm height 0.1mm width \the\marginparwidth}%
\hbox to 0mm{\hss$\rightarrow$\hspace*{-0.5mm}}$\\\relax\raggedright #1}}}
\newcommand{\ow}[1]{\widetilde{#1}}

\title[The Ihara Zeta Function for infinite graphs]{The Ihara Zeta Function for
infinite  graphs}

\author[Lenz]{Daniel Lenz} \address{Mathematisches Institut \\Friedrich Schiller
Universit{\"a}t Jena \\07743 Jena, Germany } \email{daniel.lenz@uni-jena.de}

\author[Pogorzelski]{Felix Pogorzelski}\address{Mathematisches Institut
\\Universit\"at Leipzig \\04109 Leipzig, Germany }
\email{felix.pogorzelski@math.uni-leipzig.de}

\author[Schmidt]{Marcel Schmidt}\address{Mathematisches Institut \\Friedrich
Schiller Universit{\"a}t Jena \\07743 Jena, Germany }
\email{schmidt.marcel@uni-jena.de}

\date{\today}

\begin{document}

\begin{abstract}
We put forward the concept of measure graphs. These are   (possibly
uncountable) graphs equipped with an action of a groupoid and a
measure invariant under this action. Examples include finite graphs, periodic
graphs, graphings and percolation graphs. Making use
of Connes' non-commutative integration theory we construct a Zeta
function  and present a determinant formula for it. We further
introduce a notion of weak convergence of measure graphs and show
that our construction is compatible with it. The approximation of
the Ihara Zeta function via the normalized version on finite graphs
in the sense of Benjamini-Schramm follows as a special case. Our
framework not only unifies  corresponding earlier results occurring
in the literature. It likewise provides extensions to rich new
classes of objects such as percolation graphs.
\end{abstract}

\maketitle

\tableofcontents

\section*{Introduction}

 The theory of Zeta functions of finite graphs is a well established topic connecting
various branches of mathematics, see e.g.\@ the monograph  by Terras
\cite{Terras}. Here, the Zeta function  comes about as function
storing information on the number of loops in the graph. More
specifically, it is essentially given  as exponential of a power
series whose $n$th-coefficient is determined  by the number of loops
of length $n$.  In contrast, Zeta functions on infinite graphs are
much less understood. In fact, for general infinite graphs it is not
even clear how to define a Zeta function in the first place as - due
to the infiniteness of the graph -  there are infinitely many loops
of each length.

Recent years have seen quite some interest in Zeta functions on
infinite graphs.  Indeed, for certain periodic graphs an ad-hoc
definition of the Zeta function has been given  by Clair /
Mokhtari-Sharghi in \cite{CMS01} and for certain specific examples
it has been investigated  how to define a Zeta function via suitable
approximations by Grigorchuk / Zuk \cite{GZ}, Clair /
Mokhtari-Sharghi \cite{CMS} and Guido / Isola / Lapidus
\cite{GIL08,GIL09}. The authors of \cite{GIL08} note as a main
motivation for their study that there are only very few infinite
graphs for which a Zeta function is defined. Also, for the two
dimensional integer lattice a Zeta function has been defined and
computed by Clair  in \cite{Clair} and for general regular graphs
with a transitive group action a Zeta function has  been defined and
studied in its connection to heat kernels in Chinta / Jorgenson /
Karlsson \cite{CJK13}. A recent approach for a class of infinite
weighted graphs can be found in \cite{Dei14}.

 Roughly speaking these works offer two
different solutions to  deal with the mentioned problem of
infiniteness of number of loops of a given size: One solution is to
suitably approximate the infinite graph by finite graphs and show
convergence of the Zeta functions of the finite graphs
\cite{GZ,CMS,GIL08,GIL09}.  The other solution amounts to only
counting  the loops at finitely many special vertices
\cite{Clair,CJK13} as is very natural in the presence of symmetries
in the graph.

While these offer very convincing solutions in specific cases, there
is so far  no general procedure on how to associate a Zeta function
to a graph or how to approximate it and there is no closed formula
for a Zeta function on a general graph. This is the starting point
for our paper. Our main aims  are the following:

\begin{itemize}
\item  To associate an Ihara type Zeta function to a large class of graphs (called measure graphs below)
containing finite, periodic, percolation graphs and graphings as subclasses.

\item To provide a closed formula for this Zeta function via determinants on von Neumann algebras.

\item To study the continuous dependence of this  Zeta function on the underlying measure graph.
\end{itemize}

The corresponding  results are all  new in the general context
provided here. On the one hand, they give a systematic and unified
foundation for the works mentioned above. On the other hand, they
also extend the framework studied in the literature in various ways.
For instance, graphs obtained by vertex percolation on Cayley graphs
fall into the studied class. Consequently, one obtains the - to the
knowledge of the authors - first approach to define the Ihara Zeta
function for percolation graphs.

\smallskip

As part of our investigation
\begin{itemize}
\item we put forward a notion of weak convergence of measure graphs.
\end{itemize}
Weak convergence of measure graphs may be of independent interest.
It  contains the concept of Benjamini-Schramm convergence for finite
graphs as a special case.  Moreover,  it also allows for convergence
of infinite graphs. For example, we show that percolation graphs
with weakly convergent probability laws
are weakly convergent as measure graphs. 
Another advantage of weak convergence of measure graphs  concerns
the description of the limiting object. More precisely, when dealing
with  sofic Cayley graphs or periodic graphs  we are able to
directly obtain the original graph as the limit object from a weakly
convergent sequence  of finite measure graphs.

In our context, weak convergence of measure graphs allows us to
settle the issue of the continuous dependence of the Zeta function on the underlying measure graph. More
specifically, we establish that weak convergence of measure graphs
implies convergence of the corresponding Zeta functions. From this
continuity result we obtain the previously known approximation
results for Zeta functions via finite graphs as a special case.
Moreover, as a complete novelty,  we obtain from this continuity
result the  convergence of the Zeta functions associated with a
weakly convergent sequence of (infinite) percolation graphs.


To achieve  the mentioned aims and results we
introduce  the concept of a measure graph. Indeed, setting up the
framework centered around measure graphs can be seen as the main
task  in our approach. The idea behind it is simple: Measure graphs
provide a measure on the graph and this measure satisfies an
invariance property reflecting the symmetries of the graph.  This
then  allows one to calculate an `averaged number of loops of a
given size' by counting in each vertex the loops of this size and
then average this function via the given measure. In this way, loops
at all vertices are taken into account and at the same time one ends
up with a finite number. On the technical level quite some care is
required, in particular, in order to  implement the invariance of
the measure. To do so, we rely on Connes' non-commutative
integration theory \cite{Connes-1979}.

Our set-up may be of interest for other questions as well. For
example it may be useful for dealing with random Schr{\"o}dinger
operators on graphs.



The present paper is organized as follows. 
We present our framework in Section~\ref{Framework}.  In particular,
there we  introduce the concept of   measure graph $(G,M)$
consisting of a measurable graph $G$ together with a measure $M$ on
its vertex set. In order to formulate the invariance property of the
measure we will need the action of a groupoid $\calG$ on $G$. The
assumptions required for this action lead us to the concept of a
graph over a groupoid.

We then use these ingredients to   introduce the concept of the
Ihara Zeta function of a measure graph $(G,M)$ in Section
\ref{Zeta}. The Zeta function is put forward  as the exponential of
a power series. The coefficients of the power series are determined
via Connes' non-commutative integration theory.  In this way we
effectively obtain these coefficients as integrals over the space of
vertices.

Non-commutative integration  theory also allows us to introduce von
Neumann algebras associated to measure graphs. This is discussed in
Section~\ref{Calculating-vertices} and may, again, be of independent
interest in further studies as well.  We make use of these tools to
prove that the Zeta function can be calculated via a determinant of
an operator on the vertices. Specifically, with notation introduced
below,  Theorem \ref{Determinant formula} gives for each measure
graph $(G,M)$ the following.

\medskip

\textbf{Determinant formula.} $Z_{(G,M)}(u) ^{-1} =
(1-u^2)^{-\chi_{(G,M)}} {\rm det}_\tau (I - u A_G + u^2 Q_G).$

\medskip

Of course, the use of determinants of non-positive operators
requires some care.  Here, we essentially use the determinant
provided in \cite{GIL08,GIL09}. For positive invertible operators,
this notion coincides with the famous Fuglede-Kadison determinant
\cite{FK}, see also  \cite{Lu}. For possibly non-invertible
 operators, one has to deal with
singularities. Results for such elements in a von Neumann algebra
associated with some countable, amenable group have recently been
proven by Li / Thom in \cite{LT}.


\smallskip

In Section \ref{essRegular} we have a look at the case that the
vertex degree is constant. In this case the determinant formula can
be considerably  simplified.  It  can then  be
expressed via the so-called integrated density of states. This is
the content of  Theorem \ref{determinant-formula-regular}. The
formula proven  in the theorem has been used  in \cite{GZ} to define
a Zeta function. It also has recently been obtained in \cite{CJK13} via
an analysis of Bessel functions and heat kernels.

\medskip

Section~\ref{Result} is concerned with the convergence of measure graph
sequences and the convergence of the corresponding sequence of Ihara Zeta functions.
Our notion goes beyond weak convergence of finite graphs in the sense of Benjamini-Schramm.
Concerning the convergence of the underlying Ihara Zeta functions,
we obtain in Theorem~\ref{theorem:convergence of zeta functions}:


\medskip

\textbf{Continuity result.} If the measure graphs $(G_n,M_n)$ with uniform
vertex degree bound converge
weakly to the measure graph $(G,M)$, then the
$Z_{(G_n,M_n)}$ converge to $Z_{(G,M)}$ compactly around zero.

\medskip

This result covers all the
earlier results on convergence of graph Zeta functions given in
\cite{CMS,GZ,GIL08,GIL09}. It even strengthens them by providing an
interpretation of the limit as the Zeta function of a graph.
Even more, we derive the corresponding approximation
for connected, finite graphs converging weakly
in the sense of Benjamini-Schramm, cf.\@ Theorem~\ref{theorem:weakly convergent graphs}.
Another new application is the convergence of the Ihara Zeta functions
associated to sequences of percolated Cayley graphs
with their probability laws $\mathbb{P}_n$ being weakly
convergent, cf.\@ Theorem~\ref{theorem:percolationconv}.

\medskip

Furthermore, the result can be used to provide a (rather large)
class of new examples on which the Zeta function can be obtained via
approximation.
This is discussed in Section \ref{Actions}, where we study graphs
allowing for a proper action of a sofic group with finite covolume.
We explicitly construct a sequence of finite graphs
converging towards the original graph as measure graphs. This is
the content of Theorem~\ref{thm:soficapprox}.
The convergence on the level of Zeta functions is immediate
from the previous section.

\smallskip
%
The considerations of the present work are phrased within the
measurable category. However, in  prominent classes of  examples we
often  have some additional topological information at hand. This
and more will be  addressed  in a companion  paper
\cite{LenzPS-162}.

\vspace{.85cm}

{\small

\textbf{Acknowledgements.} The authors gratefully acknowledge most
inspiring discussions on Zeta functions with Anders Karlsson and
Michel Lapidus and on  non-commutative integration theory  with
Norbert Peyerimhoff, Peter Stollmann  and Ivan Veseli\'{c}. The
authors thank Anton Deitmar for valuable remarks on groupoids and on
the notation in the present manuscript. F.P.\@ thanks Damien
Gaboriau for a very enlightening discussion at the IAS at Hebrew
University and for addressing the question of percolation measured
graphs. Moreover, the authors would like to thank the referees for
their careful reading of the manuscript and the corresponding
suggestions. F.P.\@ expresses his thanks for support through the
German National Academic Foundation (Studienstiftung des deutschen
Volkes). M.S. has been financially supported by the
Graduiertenkolleg 1523/2 : Quantum and gravitational fields. }

\section*{List of essential  pieces of notation}
Here we present a list of the  main pieces of notation together with
a short explanation and  the number of the pages on which  they are
introduced.

\begin{itemize}

\item $G = (V,E)$: graph with vertex set $V$ and edge set $E$
(Page \pageref{graph}).

\item $V^{(2)}$: set of pairs of vertices in the same connected
component (Page \pageref{vtwo}).

\item $a_G$: the adjacency matrix of $G$ (Page \pageref{adjacency})

\item $B_r^G (x)$ rooted graph with root $x$ induced from the
$r$-ball around $x$ (Page \pageref{brx}).

\item $\calG$: groupoid (Page \pageref{groupoid}).

\item $\nu$: transversal function  on $\calG$ (Page \pageref{transversal}).

\item $\eta$: the canonical random variable assigning each vertex
the mass one (Page \pageref{canonical}).

\item $\Omega$ units of the groupoid (Page \pageref{unitspace}).

\item $\av$ averaging function i.e. function satisfying $\nu \ast \av=
1$ (Page \pageref{averaging}).

\end{itemize}

\section{The framework of measure graphs}\label{Framework}
In this section we present the notation and concepts used throughout
the paper. In particular, we introduce our concept of measure graphs
over groupoids. These consists of a (not necessarily countable)
measurable graph together with a measure satisfying some invariance
property. The invariance property is phrased via a groupoid. More
specifically, the   basic pieces of data used in our setting are the
following:

\begin{itemize}

\item A measurable graph $G = (V,E)$.

\item A measurable groupoid $\mathcal{G}$  such that the graph is a space over
the groupoid in the sense of Connes.

\item A measure $M$ which is invariant with respect to the groupoid.

\item an averaging function $\av$ providing  a connection between the
groupoid and the graph.

\end{itemize}

These four pieces of data are discussed in the subsequent subsections.

\subsection{Graphs} Here we introduce  the concept of graphs used
in the sequel. These will be undirected graphs with uniform bounded
vertex degree and without loops.

\bigskip

By a  \emph{graph} we mean a  tuple $G = (V,E)$ consisting of a set
 of \emph{vertices} $V \neq \emptyset$ and a set of \emph{edges} $E
\subseteq V\times V$ such that the following holds:\label{graph}

\begin{itemize}

\item Whenever  $(x,y)$ belongs to $E$ then so does $(y,x)$.

\item There is no $x\in V$ such that $(x,x)$ belongs to $E$.

\item There is a $D>0$ such that the cardinality of $\{ y\in V :
(x,y)\in E\}$ is bounded by $D$ for any $x\in V$.

\end{itemize}

\begin{Remark}
We  emphasize that we do not put any restrictions
on the cardinality of $V$ nor $E$.
\end{Remark}

\medskip

Let $G = (V,E)$ be a graph.  For given $x \in V$ we call the pair
$(G,x)$ a \emph{rooted} graph with \emph{root} $x$. If $(x,y) = e
\in E$ we write $x \sim y$ and call $x = o(e)$ the \emph{origin} and
$y = t(e)$ the \emph{terminal vertex} of $e$. For an edge $e =
(o(e),t(e))$ we define the \emph{reversed edge} via $\bar{e} =
(t(e),o(e))$. Two edges $e,f$ are called \textit{incident} if
$\{t(e), o (e)\} \cap\{ t(f), o (f)\}$ consists of exactly one
element.

The \emph{vertex degree} at $x$ is the number of edges with origin
$x$. It will be denoted by $\mbox{deg}(x)$. In this way, $\mbox{deg}
$ becomes a function from $V$ to the non-negative integers.


A \emph{path} is a finite sequence of edges $(e_1,\ldots,e_n),$ such
that $o(e_{i+1}) = t(e_i)$ for each $i = 1,\ldots,n-1$. The number
of edges occurring in a path $P$ is called its \emph{length} and is
denoted by $\ell(P)$. Two vertices $x,y \in V$ are said to be
\emph{connected}, if there exists a path $(e_1,\ldots,e_n),$ such
that $o(e_1) = x$ and $t(e_n) = y$. If $x,y \in V$ are connected
their \emph{combinatorial distance}, $d(x,y),$ is the length of
the shortest path connecting them. If $x$ and $y$ are not connected
we set $d(x,y)~=~\infty$.

A \textit{connected component} in a graph is a maximal set of
vertices such that the combinatorial distance between any two
elements of this set is finite.  For an $x\in V$ the connected
component containing $x$ is  the set $V(x) $ of all vertices which
are connected with $x$. We denote the induced subgraph by  $G(x) =
(V(x),E \cap [V(x) \times V(x)])$. For vertices $x,y\in V$, we will
write $x \approx y$ if $x$ and $y$ belong to the same  connected
component.

Further, \label{brx} for $r \in \NN$, we let $B^G_r(x)$ denote the
graph with root $x$ which is induced by $G$ when restricting the
vertex set to the combinatorial $r$-ball around $x$.  Note that by
assumption on the uniform boundedness of the degree,
$B^G_r(x)$ is finite and the graph $G(x)$
is at most countable.

The \textit{radius}, $\varrho(G,x)$ of a finite connected  graph $G$
with root $x$ is the maximal distance of a vertex from the root,
i.e.

$$\varrho (G,x) = \max\{d(y,x) : y \in V(G)\}.$$

\medskip

Any graph comes naturally with a certain product space and a
canonical function on it. This is discussed next. Let $G$ be a
graph. Then, we define \label{vtwo}
$$V^{(2)}:= V^{(2)}_G:=\{ (x,y)\in V\times V: G(x) = G(y)\}\subset V\times V.$$
On $V^{(2)}$ there is the canonical function, called
\label{adjacency} \textit{adjacency matrix} of $G$, defined via
$$ a_G : V^{(2)} \to \{0,1\}, \;  a_G (x,y) =1 \text{ if } x\sim y \text{ and }a_G (x,y)=0 \text{ else.} $$
So, in particular, we have that
$$E = a_G^{-1} (1)\subset V^{(2)}.$$
We denote the restriction of $d$ to $V^{(2)}$ by $d$ again.

Whenever $V$ carries a $\sigma$-algebra, then $V^{(2)}$ becomes a
measurable  space with the $\sigma$-algebra induced by the product
$\sigma$-algebra on $V\times V$ and so does its subset $E$.

The real numbers and the complex numbers (and subsets thereof) will always be equipped with the Borel-$\sigma$-algebra
generated by the open subsets. Moreover, we will need the extended
positive half-axis $[0,\infty] = [0,\infty)\cup\{\infty\}$. It will
be equipped with the $\sigma$-algebra generated by the
Borel-$\sigma$-algebra on $[0,\infty)$.

Two graphs $G_1 = (V_1,E_1)$ and $G_2 = (V_2,E_2)$ are called
\textit{isomorphic} if there exists a bijective map $\phi : V_1\to
V_2$ with $x\sim y$ if and only if $\phi (x) \sim \phi (y)$. This
map is then called a \textit{graph isomorphism}. Two finite rooted
graphs are  called isomorphic if there exists an isomorphism between
them which maps the root of one into the root of the other graph.
Obviously, isomorphy is an equivalence relation on all finite rooted
graphs. For $r\geq 0$, we denote by $\cA_r^D$ the set of all
equivalence classes of finite connected rooted graphs with vertex
degree bounded by $D$ and radius equal to $r$. By $\cA^D$ we
denote the union over $r\geq 0$ of  all $\cA_r^D$.
 Due to the boundedness assumption on the degree, this is a
countable set. We will equip it with the discrete topology and the
induced $\sigma$-algebra (both of which agree with the power set).
For any $r\geq0$ we let $\pi_r$ be the map
$$\pi_r : V\to \cA^D,\, x\mapsto [B^G_r (x)],$$
where  $[ \cdot ]$ denotes the class of a rooted graph modulo
isomorphy. For a given $\alpha \in \mathcal{A}^D$ and a set of vertices $\ow{V}\subseteq V$ we let
$$\ow{V}_\alpha := \ow{V} \cap \pi_{\rho(\alpha)}^{-1}(\alpha) = \{x \in \ow{V}\, : \, \pi_{\rho(\alpha)}(x) = \alpha\}$$
be the set of all vertices in $\ow{V}$ whose $\rho(\alpha)$-ball is
isomorphic to $\alpha$. Here, of course, the choice $V = \ow{V}$ is
possible resulting in the set  $V_\alpha$.

If for a graph $G$ and $r \geq 0$, the subgraph of $G$ induced by
the ball $B_r^G(x)$ around a vertex $x$ is a  representative  of the
class $\alpha \in \mathcal{A}^D$, we will also write $B_r^G(x) \in
\alpha$.

\medskip

We will be interested in graphs carrying a $\sigma$-algebra so that
the graph is locally constant in a certain sense.

\begin{Definition}[Measurable graph] A pair $(G, \mathcal{B})$
consisting of a graph $G=(V,E)$ and a $\sigma$-algebra $\mathcal{B}$
on $V$ is called a \textit{measurable graph} if the following
conditions are satisfied:

\begin{itemize}
\item For any $r\geq0$ the map $\pi_r : V\to \cA^D,\, x\mapsto [B^G_r
(x)],$ is measurable.

\item The adjacency matrix $ a_G : V^{(2)}\to
\{0,1\}$ is measurable.

\item For any  two measurable $a,b : V^{(2)} \to
[0,\infty]$, the \textit{matrix product}
$$ a\ast b : V^{(2)} \to [0,\infty], \;  (a\ast b) (x,y)
:=\sum_{z \approx x} a(x,z) b(z,y)$$ is measurable.
\end{itemize}
\end{Definition}

\begin{Remark} We  do not require the measurability of $V^{(2)}$ as a
subset of $V\times V$ (equipped with the product $\sigma$-algebra).
We rather work directly with the $\sigma$-algebra induced on
$V^{(2)}$ by the product $\sigma$-algebra.
While this does not play a role in the examples discussed below it
may well be an advantage in further studies. In particular, it will be relevant in the already mentioned companion paper \cite{LenzPS-162}. 
\end{Remark}

\medskip

In a measurable graph certain basic quantities are automatically
measurable. This is collected in the next proposition. It will be
used tacitly in the sequel.

\begin{Proposition}[Measurability of basic quantities] \label{Measurability-basic-quantities} Let $(G,\mathcal{B})$ be a measurable graph.
Then, the following assertions hold:

(a) The  combinatorial distance  $d : V^{(2)}\to[0,\infty)$ is
measurable.

(b) The diagonal $\{(x,x) : x\in V\} \subset V^{(2)}$ is measurable.

(c) For any $r,s,t\geq 0$ and $\alpha,\beta \in \cA^D$ the set
$$\{(x,y)\in V^{(2)} : \pi_r (x) = \alpha, \pi_s (y) = \beta,
d(x,y)\leq t\}$$ is measurable.
\end{Proposition}
\begin{proof} Clearly, (b) is a direct consequence of (a).
Similarly, (c) is a direct consequence of (a) and the measurability
of the $\pi_u$, $u\geq 0$. Thus, it suffices to show (a): By
measurability of the adjacency matrix $a_G$ and the matrix product
all  powers $a_G^n$, $n\in \NN$,  (defined inductively via $a_G^1 :=
a_G$ and $a_G^{n+1} := a_G \ast a_G^n$) are measurable. Now,
clearly, $d(x,y) = 1$ holds if and only if $a_G (x,y) = 1$ holds
and, for $n\geq 1$, $d(x,y) = n$ holds if and only if both $a_G^n
(x,y) \neq 0$ and $a_G^k (x,y) =0$, $k=1,\ldots, n-1$ hold. This
shows measurability of the sets $\{ (x,y)\in V^{(2)}  : d(x,y) =
n\}$ for $n\in \NN$. This, then implies  measurability of
$$\{ (x,y)\in V^{(2)} :
d(x,y) = 0\} = V^{(2)} \setminus \bigcup_{n\in \NN}\{ (x,y)\in
V^{(2)} : d(x,y) = n\}.$$ This finishes the proof.
\end{proof}

\begin{Proposition}[Measurability of product of matrix with a
function] \label{Measurability-product} Let  $(G,\mathcal{B})$ be a measurable graph. Then, for
any measurable $a : V^{(2)}\to [0,\infty]$ and any
measurable $F : V\to [0,\infty]$ the map
$$ a\widetilde{\ast} F  :  V\to [0,\infty], \:   x\mapsto \sum_{y\approx x} a(x,y) F(y)$$
 is measurable. In particular, the vertex degree $\mbox{deg} :
 V\to [0,\infty)$ is measurable.
\end{Proposition}
\begin{proof} As the $\sigma$-algebra on $V^{(2)}$ is
the restriction of the product $\sigma$-algebra the maps
$$p_1 : V^{(2)}\to V,\:  (x,y) \mapsto x,\;\:\mbox{and}\;\:
\:  j : V\to V^{(2)}, \:  x\mapsto (x,x)$$ are measurable. Hence,
$$ a \widetilde{\ast} F = (a \ast (F\circ p_1)) \circ j$$ is measurable
 as a composition of measurable functions. Now, the last statement follows from $\mbox{deg} = a_G
\widetilde{\ast} 1$.
\end{proof}

%
%

The vertex set of a measurable graph can be disjointly decomposed
into measurable sets of vertices whose neighborhood looks like a
given isomorphism class $\alpha \in \mathcal{A}^D$. However, some
care has to be taken when the connected component of a vertex is
finite.  To this end, for $r\geq 0$ we define
$$\Vfr:= \{x \in V\, : \, \rho((G(x),x)) = r \}.$$

\begin{Proposition} \label{proposition:decomposition of V}
 Let $(G,\mathcal{B})$ be a measurable graph. For each $n \geq 0$ the vertex set $V$ can be written as
 $$V = \bigcup_{\alpha \in \mathcal{A}_n^D} V_\alpha \cup \bigcup_{r = 0}^{n-1} \Vfr,  $$
 where the occurring sets are measurable and pairwise disjoint. Furthermore, for each $s \leq r$ the set $\Vfr$ can be disjointly decomposed into
 $$\Vfr = \bigcup_{\alpha \in \mathcal{A}^D_s}  \Vfr_\alpha,$$
 where, of course, $\Vfr_{\alpha} = \Vfr \cap V_{\alpha}$.
\end{Proposition}
\begin{proof}
 All of the claimed properties follow easily from the measurability of the mappings $\pi_r$, $r \geq 0$, and the identity
 $$\Vfr = \bigcup_{\alpha \in \mathcal{A}^D_r} V_\alpha \setminus \left(\bigcup_{\beta \in \mathcal{A}^D_{r+1}} V_\beta  \right) =  \bigcup_{\alpha \in \mathcal{A}^D_r} \pi_r^{-1}(\alpha)  \setminus \left(\bigcup_{\beta\in \mathcal{A}^D_{r+1}} \pi_{r+1}^{-1}(\beta)  \right).$$
\end{proof}

\subsection{Graphs over groupoids and invariant measures} \label{subsection: graph over groupoid}
Here we  discuss groupoids, invariant measures on groupoids and
spaces over groupoids. In the measurable setting this can be found
e.g. in Connes' lecture notes \cite{Connes-1979}. Based on these
lecture notes a discussion of  random Schroedinger operators in this
context was then given in  \cite{LPV}. There, a specific situation
is singled out and studied in some detail. This situation is  called
'admissible setting' there. Here, we basically present a graph
version of the corresponding considerations centered around the
admissible setting in \cite{LPV}.

\medskip

\textbf{Notation.} As usual we will denote the set of  all
measurable functions on a measurable space by $\calF (X)$. The set
of non-negative measurable functions is then denoted by $\calF^+
(X)$. The set of all measures on $X$ is denoted by $\calM (X)$.

\medskip

A concise definition of a groupoid is that it is a small category in which every morphism is an isomorphism. A more detailed definition can then be given as follows, see e.g. \cite{Renault-80}.

\begin{Definition} \label{groupoid}
A triple $(\calG,\cdot, ^{-1})$ consisting of a set $\calG$, a
partially defined multiplication~$\cdot$, and an inverse
operation $^{-1}: \calG \to \calG$ is called a {\em groupoid} if the
following conditions are satisfied:
\begin{itemize}
\item $(g^{-1})^{-1}=g$ for all $g \in \calG$,
\item If $g_1 \cdot g_2$ and $g_2 \cdot g_3$ exist, then $(g_1 \cdot g_2) \cdot g_3$ and $g_1 \cdot ( g_2 \cdot g_3)$ exist as well and they are equal,
\item $g^{-1} \cdot g$ exists always and
$g^{-1} \cdot g \cdot h =  h$, whenever $g \cdot h$ exists, \item $h
\cdot h^{-1}$ exists always and $g \cdot h \cdot h^{-1} = g$,
whenever $g \cdot h$ exists. \end{itemize} \end{Definition}

A given groupoid $\calG$ comes along with the following standard
objects. The subset
$$\Omega:=\calG^{(0)} := \{ g \cdot g^{-1} \mid g \in \calG \}$$ is called the {\em
set of units}. \label{unitspace}

For $g \in \calG$ we define its {\em range} $r(g)$ by $r(g) = g
\cdot g^{-1}$ and its {\em source} by $s(g) = g^{-1} \cdot g$.
Moreover, we set $\calG^\omega = r^{-1}(\{ \omega \})$ for any unit
$\omega \in \calG^{(0)}$. One easily checks that $g \cdot h$ exists if
and only if $r(h) = s(g)$.

The groupoids under consideration will always be {\em measurable},
i.e., they possess $\sigma$-algebras  such that all relevant maps
are measurable. More precisely, we require that $\cdot:
\calG^{(2)}_G \to \calG$, $^{-1}: \calG \to \calG$, $s,r: \calG \to
\calG^{(0)}$ are measurable, where
\begin{equation*} \calG^{(2)}_G := \{ (g_1,g_2) \mid s(g_1) = r(g_2) \} \subset
\calG^2 \end{equation*} and $\calG^{(0)} \subset \calG$ are equipped
with the induced $\sigma$-algebras. Furthermore, we assume that
singletons $\{\omega\}$ with $\omega \in \mathcal{G}^{(0)}$ are
measurable as subsets of $\mathcal{G}^{(0)}$.   In this way,
$\calG^\omega \subset \calG$ become measurable sets (and thus
measurable spaces).

\medskip

\begin{Definition}[Graph over $\calG$]
 Let $\calG$ be a measurable  groupoid with the
previously introduced notations. A triple $(G, \pi, J)$ consisting
of a measurable graph $G$ with vertex set $V$   and maps $\pi$ and
$J$ is called a \textit{graph over $\calG$}  if the following
properties are satisfied.

\begin{itemize}
\item The map  $\pi: V \to \Omega$ is measurable.

\item For any $\omega\in \varOmega$ the induced graph $G^\omega$ on the vertex set
$$V^\omega: = \pi^{-1}(\{ \omega
\})$$ is countable.

\item The map  $\eta : \Omega \to
\calM (V), \eta^\omega :=\sum_{y\in\pi^{-1} (\omega)} \delta_y$, is
measurable in the sense that for any measurable $F : V
\to[0,\infty]$, the map $\Omega \to [0,\infty], \omega \mapsto
\eta^\omega (F)$,  is measurable.

\item The map
$J$ assigns, to every $g \in \calG$, a graph isomorphism  $J(g):
G^{s(g)} \to G^{r(g)}$  with the properties $J(g^{-1})=J(g)^{-1}$
and $J(g_1 \cdot g_2) = J(g_1) \circ J(g_2)$ if $s(g_1) = r(g_2)$.
\end{itemize}
The map $\eta$ is called the \textit{canonical random variable}.
\label{canonical}
\end{Definition}

%

\textbf{Notation.}  To simplify notation, we will often write $g h$
respectively $g x$ for $g \cdot h$ respectively $J(g) x$.

\medskip

Our next aim is to exhibit natural measures on these objects. The
first step in this direction is the definition of a transverse
function.

\begin{Definition}[Transversal function]\label{transversal}
Let $\calG$ be a measurable groupoid and with the  notation given
above. A {\em transversal function} $\nu$ of $\calG$ is a map $\nu:
\Omega \to \calM(\calG)$ with the following properties:
\begin{itemize}
\item The map $\omega \mapsto \nu^\omega(f)$ is measurable for every
$f \in \calF^+(\calG)$.
\item $\nu^\omega$ is supported on $\calG^\omega$, i.e.,
$\nu^\omega(\calG - \calG^\omega) = 0$.
 \item $\nu$ satisfies
the following invariance condition
\begin{equation*} \int_{\calG^{s(g)}} f(g \cdot h) d\nu^{s(g)}(h) =
\int_{\calG^{r(g)}}
f(k) d\nu^{r(g)}(k) \end{equation*} for all $g \in \calG$ and $f \in
\calF^+(\calG^{r(g)})$.
\end{itemize}

\end{Definition}

In the examples which will be considered later in the present paper,
the measures $\nu^{\omega}$ will be simple counting measures defined
on countable fibers $r^{-1}(\omega)$.

\medskip

In the next definition we introduce appropriate measures on the base
space $\Omega$ of an abstract groupoid $\calG$.

\begin{Definition}[Invariant measure]  Let $\calG$ be a measurable  groupoid
with a transversal function $\nu$. A measure $m $ on the base space
$(\Omega,\calB_\Omega)$ of units is called {\em $\nu$-invariant} (or
simply \textit{invariant}, if there is no ambiguity in the choice of
$\nu$) if
\begin{equation*} m  \circ \nu = (m  \circ \nu)^\sim, \end{equation*}
where $m \circ \nu$ is the measure on $\calG$ defined by
 $(m  \circ \nu)(f) = \int_\Omega \nu^\omega(f) dm (\omega)$ for
 measurable $f : \calG\longrightarrow [0,\infty]$
and $(m  \circ \nu)^\sim(f) = (m  \circ \nu)(\tilde f)$ with $\tilde
f(g) = f(g^{-1})$.
\end{Definition}

Analogously to transversal functions on the groupoid, we introduce a
corresponding fiberwise consistent family $\alpha$ of measures on a
graph over a groupoid.

\begin{Definition}[Random variable in the sense of Connes]
Let $\calG$ be a measurable  groupoid and $G$ a  graph over $\calG$
with vertex set $V$. A choice of measures $\xi: \Omega \to \calM(V)$
is called a {\em random variable}  with values in $G$ (in the sense
of Connes) if it has the following properties:
\begin{itemize}
\item The map $\omega \mapsto \xi^\omega(f)$ is measurable
for every $f \in \calF^+(V)$,
\item $\xi^\omega$ is supported on $G^\omega$, i.e.,
$\xi^\omega(V - V^\omega) = 0$,
\item $\xi$ satisfies the following invariance condition
\begin{equation*} \int_{V^{s(g)}} f(J(g)x) d\xi^{s(g)}(x) = \int_{V^{r(g)}}
f(y) d\xi^{r(g)}(y) \end{equation*} for all $g \in \calG$ and $f \in
\calF^+(V^{r(g)})$. \end{itemize} \end{Definition}

\begin{Remark}
Let $\xi$ be a random variable.
\begin{itemize}
\item By considering (positive) linear combinations of
functions of the form
$$F : V^2 \to [0,\infty],  (x,y)
\mapsto  f(x) g(y),$$ with measurable $f, g: V\to [0,\infty]$ and
using standard monotone class arguments, we can easily obtain
measurability of
$$V\times \Omega\to [0,\infty], (x,\omega)\to \xi^\omega (F(x,
\cdot)),$$ for any measurable $ F: V\times V\to [0,\infty]$.

\item  If furthermore   $V^{(2)}$ is a measurable subset of $V \times V$, then
any measurable function $F$ on $V^{(2)}\to [0,\infty]$
can be extended (by zero) to a measurable function on $V\times V$.
Thus, in this case we obtain measurability of $V\times
\Omega\to [0,\infty], (x,\omega)\to \xi^\omega (F(x,
\cdot))$ for any measurable $F : V^{(2)}\to [0,\infty]$.
\end{itemize}
\end{Remark}

\medskip

We can actually provide a large supply of  random variables.

\begin{Proposition}[Generating random variables] \label{Generating} Let $\calG$ be a measurable  groupoid and $G$ a  graph over $\calG$
with vertex set $V$. Then, the canonical random variable $\eta$ is a
random variable. Moreover, for  any measurable $ H : V
\to [0,\infty]$ with $H (g x) = H(x)$ for all $x \in V$
and $g\in \calG$ with $\pi (x) = s(g)$ the map
$$\xi_H : \varOmega \to \calM(V),\, \xi_H^\omega
:=\sum_{y\in \pi^{-1} (\omega)} H(y) \delta_y$$ is a random variable.
\end{Proposition}
\begin{proof} The canonical random variable $\eta$ is a random variable by the
very definition of a  graph over a groupoid. Now, consider a
measurable  $H : V\to [0,\infty]$.  We have to show a measurability
and an  invariance property of $\xi_H$. Now, whenever $ F  : V\to
[0,\infty]$ is measurable, then so is $ F H : V\to [0,\infty]$.
Hence,
$$\omega \mapsto \xi_H^\omega  = \eta^\omega (HF)$$
is measurable by the assumption on $\eta$. The invariance of $\xi_H$
is clear from the invariance property of $H$.
\end{proof}

\begin{Remark} If the diagonal $\{(x,x) : x\in V\}$ is a measurable
subset of $V \times V$, then any random variable arises in this way.
Indeed, in the case  the previous remark applied with $F$ the
characteristic function of the diagonal, gives easily that for any
random variable $\xi$ the function $H(x) = \xi^{\pi(x)}
(F(x,\cdot))$ is measurable and invariant. By construction $H$ is
the density of $\xi$ with respect to $\eta$.
\end{Remark}

\medskip

Whenever $G$ is a graph over a groupoid $\calG$ with transverse
$\nu$, we use the following notation for the \textit{convolution} of
a $w \in \calF^+(V )$ with respect to $\nu$
\begin{displaymath}
\nu\ast w (x) := \int_{\calG^{\pi(x)}} w (g^{-1} x) d\nu^{\pi(x)}
(g) \quad \text{ for } x \in V .
\end{displaymath}

A crucial fact about the integration of random variables is given in
the following lemma, essentially contained in (the proof of) Lemma
III.1 in \cite{Connes-1979}, see Lemma 2.9 of \cite{LPV} as well.

\begin{Lemma}\label{ind}
 Let $\calG$ be a measurable groupoid with transversal function $\nu$
 and $\nu$-invariant measure $m $. Let  $G $ be a graph over $\calG$.  \\
 (a) The integral
 $\int_\Omega \nu^\omega(f)\,dm (\omega)$ does not depend
 on $f \in \calF^+ (\calG)$, provided $f$ satisfies
 $\nu(\tilde f) \equiv 1$. \\
 (b) For a given random variable $\xi$ with values in $G $ the
 integral $\int_\Omega \xi^\omega(\av)\, dm (\omega)$ does not depend on
 $\av$, provided $\av\in \calF^+ (V )$ satisfies $\nu\ast \av \equiv 1$.
\end{Lemma}

The previous lemma gives the possibility to define an integral over
a random variable.

\begin{Definition}[Integral of a random variable] \label{Integral-random-variable} Let $\calG$ be a measurable groupoid
with transversal function $\nu$
 and $\nu$-invariant measure $m $. Let  $G $ be a
 graph over $\calG$ with vertex set $V$  such that there exists a $\av\in \calF^+
(V)$ satisfying $\nu\ast \av \equiv 1$.
Then, the integral over the random variable $\xi$ is denoted as
$\int \xi$ and  defined via
$$\int \xi= \int_\Omega \xi^\omega (\av) dm (\omega),$$
where $\av$ is any (not necessarily strictly positive) element of $\calF^+ (V)$ satisfying $\nu\ast \av
\equiv 1$.
\end{Definition}

The functions $\av$ appearing in the previous definition deserve a
special name.

\begin{Definition}[Averaging function]\label{averaging}
Let $\calG$ be a measurable groupoid with transversal function $\nu$ and $\nu$-invariant measure $m $. Let  $G$ be a
 graph over $\calG$ with vertex set $V$. Then, any  function  $\av\in \calF^+
(V)$ satisfying $\nu\ast \av \equiv
 1$ is called an \textit{averaging function}.
\end{Definition}

We note that existence of an averaging function is quite essential
to our approach. In fact, it is only via these averaging functions
that the integration of random variables could be  defined above.
 On the conceptual level  the existence of an averaging function
$\av$ can be understood as giving a way to relate  $\calG$ and $G$ via
a map  from $\calF (\calG)$ to  $\calF (V).$
 Namely, every  averaging function
$\av$ gives rise to the fiberwise defined map $$q = q_{\av}\colon
\calF(\calG) \to \calF(V)$$ with
\begin{equation*}
q(f)(x) := \int_{\calG^{\pi(x)}} \av(g^{-1}x) f(g) d\nu^{\pi(x)}(g)
\end{equation*}
for all $x \in V$. Note the defining property of  $\av$ implies that
the map $q$ satisfies $q(1_\calG) = 1_V$ (see Section 2  of
\cite{LPV} for further discussion).

\subsection{Measure graphs}
After all these preparations we can now introduce the main concept
of our study. We will need one   further assumption  on countability
of generators of the $\sigma$-algebras in question in order to apply
the integration theory developed in \cite{Connes-1979}.

\begin{Definition}[Measure graph over a groupoid] \label{admissett}
A pair $(G,M)$ consisting of a measurable  graph $G$ and a measure
$M$ on the vertices of $G$ is called a  \textit{measure graph over
the  groupoid $\calG$}  or just \textit{measure graph} for short if
the following properties hold:

\begin{itemize}
\item $\calG$ is a measurable groupoid and  $G$ is a graph over $\calG$.

\item $\calG$ admits a transversal function $\nu$ together with an
invariant measure $m$  and $M = m \circ \eta$ (where $\eta$ is the
canonical random variable).

\item The $\sigma$-algebras of both
 $V$ and  $\varOmega$ possess a countable basis of sets, all of
which have finite measure w.r.t.~$M$ (respectively w.r.t.~$m $).

\item
There exists a strictly positive averaging   function $\av$.

\item The canonical random variable $\eta$ is integrable, i.e.
$\int \eta = \int_\varOmega \eta^\omega(\av) dm (\omega) = \\
\int_V \av(x) dM (x) < \infty$.

\end{itemize}
The measure $m$ will be called the \textit{invariant measure
underlying the measure $M$} and $\nu$ will be called the
\textit{underlying transversal} on the groupoid.
\end{Definition}

\begin{Remark}
The above assumptions ensure that our setup falls within the general
framework of \cite{Connes-1979} (as specified in  Definition 2.6 of
\cite{LPV}). This will enable us to use the associated von Neumann
algebras and to use integration theory. The countability assumptions
on the $\sigma$-algebras mean that the condition of being 'propre'
in the sense of \cite{Connes-1979} is satisfied. Integrability of
the canonical random variable will yield a $\sigma$-finite weight
(even a trace) on the corresponding von Neumann algebra and strict
positivity of the averaging function will entail faithfulness of
this weight.
\end{Remark}

For our further considerations two features of measure graphs will
be crucial (see below):

\begin{itemize}
\item
Any measure graph comes with a natural Hilbert space $L^2 (V,M)$. We
will be concerned with operators on this Hilbert space.

\item Any measure graph naturally allows for
integration of random variable as it possesses  by its very
definition an averaging function.
\end{itemize}

\subsection{Some examples of  measure graphs}\label{Examples} In this section we discuss some  examples of
measure graphs over groupoids. In particular, we discuss the integration of their random variables.

\subsubsection{Finite graphs} Let $G$ be a finite connected  graph
with vertex set $V$. Equip $V$ with the discrete $\sigma$-algebra, which is finite.
 Let $\calG$ be the trivial groupoid (consisting only of one element $e$) acting on $G$
as the identity. Its set of units $\Omega =\{e\}$ consists of only one point and $\pi : V\to \{e\}$ is measurable.  Let $\nu$ be the  measure giving value $1$ to  the  point $\{e\}$. Then, the measure
$m$  giving value $1/C$ to the point $\{e\}$ is invariant and the constant function $\av = 1$ satisfies $\nu \ast \av \equiv 1$. The measure $M = m \circ \eta$ on $V$ is the counting measure on the vertex set normalized by $1/C$. In this way, $(G,M)$ is a measure graph over $\calG$. Furthermore, all random variables are given by functions on the vertex set and their integral equals the sum of their values normalized by $1/C$.

\subsubsection{Periodic graphs}

We first recall the notion of a periodic graph.
\begin{Definition}[Periodic graph] \label{periodic graph}
 A pair $(G,\Gamma)$ consisting of a countable infinite connected graph $G = (V,E)$ and a countable subgroup $\Gamma$ of the automorphism group of $G$ is called {\em periodic graph}, if
 \begin{itemize}
  \item $\Gamma$ acts on $V$ properly, i.e., for each $x \in V$ the stabilizer $\Gamma_x = \{\gamma \in \Gamma\, : \, \gamma x = x\}$ is finite.
  \item $\Gamma$ acts on $V$ with finite co-volume, i.e., there is a set $F \subset \Gamma$ containing exactly one representative of each of the classes of $V/\Gamma$
  satisfying $\sum_{f \in F} \frac{1}{|\Gamma_f|} < \infty$.
 \end{itemize}
The subset $F$ is called a  {\em fundamental domain}.
\end{Definition}

Let $(G,\Gamma)$ be a periodic graph.  We equip its vertex set $V$
with the discrete $\sigma$-algebra, which is countably generated. Let
$\calG$ be given by $\Gamma$ with the corresponding action on $G$.
The space of units $\Omega$ consists exactly of the neutral element
$e$ of $\Gamma$ and the map $\pi : V\to \{e\}$ is clearly
measurable. A transversal function  $\nu = \nu^e$ is given by the
counting measure on $\Gamma$. The measure $m$ giving mass $1$ to the
set $\{e\}$ is invariant. Fix a fundamental domain $F$ of finite co-volume
of the action of $\Gamma$ and define the function $\av_0$ on $V$ by
$$\av_0 (x) := \frac{1}{ |\Gamma_x|} 1_{F} (x).$$
Here $|\Gamma_x|$ denotes the number of elements of the
stabilizer of $x$ and $1_F$ is the characteristic function of $F$.
A short calculation shows that $\nu \ast \av_0 (x) = 1$ for
all $x\in V$, i.e., $\av_0$ is an averaging function. It
is not strictly positive,  though. However, for any $\gamma \in
\Gamma$, the function $\av_\gamma(\cdot) := \av_0 (\gamma \cdot)$ is
  an averaging function as well. So, if we chose a sequence
$(c_\gamma)_{\gamma \in \Gamma}$ of positive numbers with $\sum
c_\gamma = 1$, then
$$ \av :=\sum_{\gamma \in\Gamma } c_\gamma \av_\gamma$$
will be a strictly positive averaging function. The measure $M = m \circ \eta$
is  just the counting measure on the vertex set. In this
way, $(G,M)$ is a measure graph over $\calG = \Gamma$.

Recall the definition of the random variable $\xi_F$ for a
$\Gamma$-invariant function $F:V \to [0,\infty]$ (see
Lemma~\ref{Generating}). By Lemma~\ref{ind} its integral satisfies
$$\int \xi_F = \int_V F \av {\rm d} M = \int_V F \av_0 {\rm d} M = \sum_{x \in F} \frac{F(x)}{|\Gamma_x|}.$$
%

%
%

\subsubsection{Percolation graphs} \label{expample:percolation graphs}

Let $\Gamma$ be a finitely generated group with a symmetric set of generators $I = \{g_1,\ldots,g_l,g_1^{-1},\ldots,g_l^{-1}\}$ and let $e$ be its neutral element.  By $G_0$ we denote  the corresponding right Cayley graph, i.e., $\gamma, \gamma' \in \Gamma$ are neighbors in $G_0$ if and only if there exists $g \in I$ such that $\gamma = \gamma' g.$ The left-action of $\Gamma$ on itself induces graph isomorphisms of $G_0$.  We perform vertex percolation on $G_0$ and construct a measure graph from its realizations. We equip the space
$$\Theta:= \{0,1\}^\Gamma$$
with the product $\sigma$-algebra (which is countably generated) and define the shift-action of $\Gamma$  by
$$\Gamma\times \Theta \to \Theta,\, (\gamma,\theta) \mapsto \gamma \theta := \theta (\gamma^{-1} \cdot).$$
Each element $\theta \in \Theta$ is identified with the graph $G_\theta = (\Gamma, E_\theta)$  whose edge set is
$$E_\theta = \{(\gamma,\gamma') \in \Gamma \times \Gamma\, : \, \gamma \sim \gamma \text{ in } G_0 \text{ and } \theta(\gamma) = \theta(\gamma') = 1\}.$$
In other words, $G_\theta$ is obtained from $G_0$ by deleting all edges that are adjacent to a vertex $\gamma \in \Gamma$ which satisfies $\theta(\gamma) = 0$.

The vertex set of the percolation measure graph is defined by $V := \Gamma \times \Theta$ and equipped with the product $\sigma$-algebra (here $\Gamma$ carries the discrete $\sigma$-algebra). We say that $(\gamma,\theta)$ and $(\gamma',\theta')$ are adjacent if and only if $\theta = \theta'$ and $\gamma \sim \gamma'$ in $G_\theta$. The so obtained graph is denoted by $G_{\Gamma,{\rm perc}}$.

\begin{Proposition}
 $G_{\Gamma,{\rm perc}}$ is a measurable graph.
\end{Proposition}
\begin{proof}
 We first show that $\pi_r :V \to \mathcal{A}^D$ is measurable. For $\alpha \in \mathcal{A}^D$ we obtain
 \begin{align*}
  \pi^{-1}_r(\alpha) &= \{(\gamma,\theta)\in \Gamma\times \Theta\,:\, B^{G_{\Gamma,{\rm perc}}}_r((\gamma,\theta)) \in \alpha \}\\
  &= \{(\gamma,\theta)\in \Gamma\times \Theta\,:\, B^{G_\theta}_r(\gamma) \in \alpha \}\\
  &= \{(\gamma,\theta)\in \Gamma\times \Theta\,:\, B^{G_{\gamma^{-1}\theta}}_r(e) \in \alpha \}\\
  &= \bigcup_{\gamma \in \Gamma} \left[\{\gamma\} \times \gamma \{ \theta \in \Theta\, : \, B^{G_\theta}_r(e) \in \alpha \}\right].
 \end{align*}
The set  $\{ \theta \in \Theta\, : \, B^{G_\theta}_r(e)\in \alpha
\}$ is a finite  union of  cylinder sets in $\Theta$. In particular,
it is measurable. Since the $\Gamma$-action on $\Theta$ is also
measurable we obtain the measurability of  $ \pi^{-1}_r(\alpha)$.

\medskip

The measurability of the adjacency matrix on $V^{(2)}$  is a consequence of

 \begin{align*}
  a_{G_{\Gamma,{\rm perc}}}^{-1}(\{1\}) &= \{ ((\gamma,\theta),(\gamma', \theta'))\in V^{(2)}\, : \, \theta = \theta' \text{ and } \gamma \sim \gamma' \text{ in } G_{\theta}\}\\
  &= V^{(2)} \cap  \{ ((\gamma,\theta),(\gamma', \theta'))\in V\times V\, : \, \text{there ex. } g \in I \text{ s.t. } \gamma' = \gamma g \} \\
  &= V^{(2)} \cap \bigcup_{(\gamma,g) \in \Gamma \times I} \left[ ( \{\gamma\} \times \Theta) \times (\{\gamma g\} \times \Theta) \right].
 \end{align*}
 Here we have used the following two facts.
 \begin{itemize}
  \item $((\gamma,\theta),(\gamma', \theta'))\in V^{(2)}$ implies $\theta = \theta ',$
  \item if $\gamma$ and $\gamma'$ belong to the same connected component in $G_\theta$ and satisfy $\gamma' = \gamma g$ for some $g \in I$ then $\gamma \sim \gamma'$ in $G_\theta.$
 \end{itemize}
It remains to show the measurability of the matrix products on
$V^{(2)}$. To this end, let measurable $a,b:V^{(2)} \to [0,\infty]$
be given. For fixed $\gamma \in \Gamma$ we define
$$a_\gamma:V^{(2)} \to [0,\infty],\, ((\gamma',\theta),(\gamma'',\theta)) \mapsto \begin{cases}
                                                                                 a(((\gamma',\theta),(\gamma,\theta))) &\text{if }  (\gamma',\theta)\approx (\gamma,\theta)\\
                                                                                 0 &\text{else}
                                                                                 \end{cases}
$$
and
$$b^\gamma:V^{(2)} \to [0,\infty],\, ((\gamma',\theta),(\gamma'',\theta)) \mapsto \begin{cases}
                                                                                 b(((\gamma,\theta),(\gamma'',\theta))) &\text{if }  (\gamma'',\theta)\approx (\gamma,\theta)\\
                                                                                 0 &\text{else}
                                                                                 \end{cases}.
$$
These functions satisfy
$$a \ast b = \sum_{\gamma \in \Gamma} a_\gamma b^\gamma.$$
Therefore, it suffices to show the measurability of $a_\gamma$ and
$b^\gamma$. Let $I = (C,\infty]$ with $C \in [0,\infty]$. We obtain
\begin{align*}
a_\gamma^{-1}(I) &= a^{-1}(I) \cap \{((\gamma',\theta),(\gamma,\theta)) \, :\, (\gamma',\theta) \in \Gamma\times \Theta \text{ and }(\gamma',\theta) \approx (\gamma,\theta)\}  \\
&= a^{-1}(I) \cap \bigcup_{\gamma' \in \Gamma} \left[\left(\{\gamma'\} \times \Theta\right) \times \left( \{\gamma\} \times \Theta\right)\right].
\end{align*}
For the last identity we used  $a^{-1}(I) \subseteq V^{(2)}$ and
\begin{align*}
 &\{((\gamma',\theta),(\gamma,\theta)) \, :\, (\gamma',\theta) \in \Gamma\times \Theta \text{ and }(\gamma',\theta) \approx (\gamma,\theta)\} \\&= V^{(2)}\cap  \bigcup_{\gamma' \in \Gamma} \left[\left(\{\gamma'\} \times \Theta\right) \times \left( \{\gamma\} \times \Theta\right)\right].
\end{align*}
This shows the measurability of $a_\gamma$ and a similar reasoning yields the measurability of $b^\gamma$.
\end{proof}

The groupoid which acts upon $G_{\Gamma,\, {\rm perc}}$ is the semidirect product $\Gamma \ltimes \Theta$ with respect to the shift-action of $\Gamma$. More precisely, we let $\mathcal{G} = \Gamma\times \Theta$ and define a partially defined multiplication on $\mathcal{G}$ by
$$(\gamma,\theta) \cdot (\gamma',\theta') = (\gamma \gamma',\theta) \text{ iff } \theta = \gamma \theta'$$
and an inverse operation by
$$(\gamma,\theta)^{-1} = (\gamma^{-1},\gamma^{-1}\theta).$$
The set of units  $\Omega = \mathcal{G}^{(0)} = \{(e,\theta)\, : \, \theta \in \Theta\}$ is tacitly identified with $\Theta$ in the obvious manner.

We now give $G_{\Gamma, {\rm perc}}$ the structure of a graph over $\mathcal{G}$. The mapping $\pi: V \to \Theta,\, (\gamma,\theta) \mapsto \theta$ is clearly measurable and the canonical random variable $\eta$ is measurable as well.  As a set, $\mathcal{G}$ coincides with the vertex set $V$ of the graph $G_{\Gamma, {\rm perc}}$. Hence, the groupoid acts naturally upon $V$ by left multiplication. The induced graph on the fiber $\pi^{-1}(\{\theta\})$ is isomorphic to $G_\theta$. Since left-multiplication with $\gamma$ induces a graph isomorphism  $G_{\gamma^{-1}\theta} \to G_\theta$ the groupoid acts by graph isomorphisms between the fibers of $\pi$. Therefore, $G_{\Gamma, {\rm perc}}$ is a graph over $\mathcal{G}.$

A transversal function of $\mathcal{G}$ is
$$\nu:\Theta \to \mathcal{M}(\mathcal{G}),\, \nu^\theta := \sum_{\gamma\in \Gamma} \delta_{(\gamma,\,\theta)}.$$
It can easily be checked that any probability measure $\mathbb{P}$
on $\Theta$ which is invariant with respect to the $\Gamma$-shift is
invariant with respect to $\nu$, see  \cite[Cor.~II.7]{Connes-1979}
as well. We fix such a measure and let $M_{\mathbb{P}} = \mathbb{P}
\circ \eta$. For any $\gamma \in \Gamma$ the function $\av_\gamma :=
1_{\{\gamma\} \times \Theta}$ is an averaging function. Taking
suitable linear combinations yields a strictly positive averaging
function $\av$. In this way, $(G_{\Gamma, {\rm perc}},M_\mathbb{P})$
is a measure graph over $\mathcal{G}$.

Recall the definition of the random variable $\xi_F$ for a $\mathcal{G}$-invariant function $F:V \to [0,\infty]$ (see Lemma~\ref{Generating}). By Lemma~\ref{ind} its integral satisfies
$$\int \xi_F  = \int_V F \av {\rm d}M_\mathbb{P} =  \int_V F \av_e {\rm d}M_\mathbb{P} = \int_\Theta F((e,\theta)) {\rm d}\mathbb{P}.$$

\subsubsection{Graphings}\label{example:graphings}
Graphings have attracted quite some attention. In particular, weakly
convergent graph sequences can be shown to have graphings as limits
(see e.g. \cite{Ele}). As discussed next, they also fall within our
framework.

Recall that a \textit{graphing}, $\mathcal{X} = (X,\mu,
\iota_1,\ldots, \iota_N)$,  consists of a compact metric space $X$
equipped with the Borel-$\sigma$-algebra  together with continuous
involutions $\iota_1,\ldots,\iota_N:X \to X$ and a Borel probability
measure $\mu$ on $X$ which is invariant under the $\iota_j$,
$j=1,\ldots, N$.

To a given graphing $\mathcal{X}$ we associate a measure graph. Let $\mathcal{J} := \langle \iota_1,\ldots,\iota_n\rangle$ be the group generated by the involutions. By  $\mathcal{R}(\mathcal{X})$ we denote the induced equivalence relation, that is,
$$\mathcal{R}(\mathcal{X}) = \{(x,y)\in X \times X\, :\, \text{there ex. } \iota \in \mathcal{J} \text{ s.t. } x = \iota y \}.$$
It is equipped with the induced $\sigma$-algebra (it is countably generated since the $\sigma$-algebra of $X$ is countably generated). The vertex set of the graph associated with $\mathcal{X}$ is $V_\mathcal{X} := \mathcal{R}(\mathcal{X})$. We say that two vertices $(x,y),(w,z) \in V_\mathcal{X}$ are connected if and only if $w = x$ and $y = \iota_j z$ with $\iota_j z \neq z$ for some $ 1 \leq j \leq n$. The so constructed graph is denoted by $G_\mathcal{X}$.

\begin{Proposition}
 $G_\mathcal{X}$ is a measurable graph.
\end{Proposition}

\begin{proof}
The proof of the measurability conditions can be shown in a similar way
as for the percolation graphs considered above.
\end{proof}

The equivalence relation $\mathcal{R}(\mathcal{X})$ is a groupoid with partially defined multiplication
$$(x,y) \cdot (w,z) = (x,z) \text{ iff }  w = y$$
and inverse operation
$$(x,y)^{-1} = (y,x).$$
As  measurable space the set of units  $\Omega = \mathcal{R}(\mathcal{X})^{(0)} = \{(x,x)\, : \, x \in X\}$ is equivalent to the Borel space $X$. In the following we tacitly identify both.

We now give $G_\mathcal{X}$ the structure of a graph over
$\mathcal{R}(\mathcal{X})$. The map $\pi: V_{\mathcal{X}} \to X,\,
(x,y) \mapsto x$ is clearly measurable. Moreover, the following
holds.  

\begin{Proposition}
The canonical random variable $\eta$ is measurable.
\end{Proposition}
\begin{proof} Let $\mathcal{J}$ be the group generated by the $\iota_j$, $j =1,\ldots,
n$. Note that by  continuity of the action for arbitrary
$\gamma,\varrho \in \mathcal{J}$ the set $\{x\in X: \gamma (x) \neq
\varrho (x)\}$ is open and its complement is closed. So, both sets
are measurable. Let now  $\gamma_1,\gamma_2, \ldots, $ be an
enumeration of  $\mathcal{J}$. By induction we can then construct for any
natural number $n$  a decomposition of $X$ into disjoint measurable
sets $X_{n,1},\ldots, X_{n,k_n}$, such that for  arbitrary fixed
$\gamma,\varrho \in \{\gamma_1,\ldots, \gamma_n\}$ and
$j\in\{1,\ldots, k_n\}$ we either have $\gamma x = \varrho x$ for
all $x\in X_{n,j}$ or $\gamma x \neq \varrho x$ for all $x\in
X_{n,j}$. Let then for any natural number $n$ and any $j\in
\{1,\ldots, k_n\}$ be  $\mathcal{J}_{n,j}$  a maximal subset of
$\{\gamma_1,\ldots, \gamma_n\}$ with $\gamma (x) \neq \varrho (x)$
for all $x\in X_{n,j}$ and $\gamma,\varrho\in \mathcal{J}_{n,j}$ with
$\gamma \neq \varrho$. Denote the characteristic function of
$X_{n,j}$ by $1_{X_{n,j}}$.
 Then, for any measurable nonnegative $F$ the function
$$X\longrightarrow [0,\infty),\; x\mapsto \eta_n^x (F):=\sum_{j=1}^{k_n} 1_{X_{n,j}}(x) \left( \sum_{\gamma\in \mathcal{J}_{n,j}} F(x,\gamma x)\right),$$
 is measurable (as $F$ is measurable and the $X_{n,j}$ are
measurable). Moreover, by maximality of the $\mathcal{J}_{n,j}$ we have
$\eta_n (F) \to \eta (F)$ pointwise. This is the desired
measurability of $\eta$.
\end{proof}

As a set, $\mathcal{R}(\mathcal{X})$ coincides with the vertex set
$V_\mathcal{X}$. Hence, the groupoid acts naturally upon
$V_\mathcal{X}$ by left multiplication. The fiber $\pi^{-1}(\{x\}) =
\{x\} \times \mathcal{J}x$ is the $\mathcal{J}$-orbit of $x$ with
distinguished root $x$. The action of the equivalence relation only
changes this root and leaves the orbit invariant. Therefore, it acts
by graph isomorphisms between the $\pi$-fibers.   In this way,
$G_\mathcal{X}$ is a graph over $\mathcal{R}(\mathcal{X}).$

 A transversal function of $\mathcal{R}(\mathcal{X})$ is
$$\nu:X \to \mathcal{M}(\mathcal{R}(\mathcal{X})),\, \nu^x := \sum_{(x,y) \approx (x,x) } \delta_{(x,y)}.$$
The measure $\mu$ on $X$ is invariant with respect to $\nu$. We let
$M_\mathcal{X} = \mu \circ \eta$. The function $\av_0 :=
1_{\{(x,x)\, : \, x \in X\}}$ is an averaging function. By shifting
with elements of $\mathcal{J}$ in the second variable and taking
suitable linear combinations it can be made strictly positive. In
this way, $(G_{\mathcal{X}},M_\mathcal{X})$ is a measure graph over
$\mathcal{G}$.

Recall the definition of the random variable $\xi_F$ for a $\mathcal{R}(\mathcal{X})$-invariant function $F:V \to [0,\infty]$ (see Lemma~\ref{Generating}). By Lemma~\ref{ind} its integral satisfies
$$\int \xi_F  = \int_V F \av_0 {\rm d}M_\mathcal{X} = \int_X F((x,x)) {\rm d}\mu.$$
%


\begin{Remark}
The considerations of this section also imply that the fractal
graphs studied in \cite{GIL09}  fall into our framework. Indeed from
the considerations of \cite{GIL09} one can easily see that they give
rise to a limit graphing. Hence, they can be treated by our methods.
We refrain from giving details here.
\end{Remark}


\section{The Ihara Zeta function of a measure graph}\label{Zeta}
In this section we discuss how a general measure graph over a
groupoid  gives immediately rise to a Zeta function. We will call
this the Ihara Zeta function. For finite and for periodic graphs it
will be shown to agree with the 'usual' Zeta function.

\subsection{The definition and basic properties}

Before introducing the Zeta function we need some further graph
theoretic notions. Let a graph  $G= (V,E)$ be given. A \emph{circle}
or \emph{closed path} is a path $(e_1,\ldots,e_n)$ such that $t(e_n)
= o(e_1)$.

\begin{Definition}[Reduced paths] Let $G = (V,E)$ be a graph.
\hspace{10cm}
\begin{itemize}
\item[(a)] A closed path $(e_1,\ldots, e_n)$ has a {\em backtrack} if
$e_{i+1}
= \bar{e}_{i}$ for some $i \in \{1,\ldots,n-1\}$. A path with no backtracking
is said to be {\em proper}.
\item[(b)] We call a circle {\em primitive} if it is not obtained by going $n
\geq 2$ times around a shorter closed path.
\item[(c)] A closed path $(e_1, \dots, e_n)$ has a {\em tail} if
there is a number $k \in \NN$ such that $\bar{e}_j = e_{n-j+1}$ for $1 \leq j
\leq k$.
\item[(d)] A closed path without backtracking or tail is called \emph{reduced}.
\end{itemize}
\end{Definition}

\begin{Definition}[Counting of closed path]
Whenever we are given a graph $G = (V,E)$, we let $N_j^G (x)$ be the
number of reduced closed paths of length $j$ starting at a given
vertex $x \in V$. Similarly, $P_j^G (x)$ will denote the number of
primitive reduced circles of length $j$ starting at $x$. If the
graph is understood from the context, we will suppress the
superscript $G$ and just write $N_j (x)$ and $P_j (x)$.
\end{Definition}

\begin{Lemma}[Averaged circle counting]
Let $\calG$ be a groupoid and let  $(G,M)$ be a measure graph over
$\calG$.  Then, for any natural number $j$ the functions $
\xi_{N_j}$ and $\xi_{P_j}$ given by
$$\Omega \ni \omega\mapsto \xi_{N_j}^\omega:=\sum_{x\in \pi^{-1} (\omega)}   N_j(x)  \delta_x \in \mathcal{M}(V),$$
and
$$\Omega\ni \omega\mapsto \xi_{P_j}^\omega:=\sum_{x\in \pi^{-1} (\omega)}   P_j(x)  \delta_x \in \mathcal{M}(V)$$
are random variables. In particular, the numbers
$$\overline{N}_j := \int \xi_{N_j} = \int_\Omega \xi_{N_j}^\omega (\av) d m (\omega) = \int_{V} N_j
(x) \av(x)  d M (x),$$
and
$$\overline{P}_j := \int \xi_{P_j} = \int_\Omega \xi_{P_j}^\omega (\av) d m (\omega)  = \int_{V} P_j
(x) \av(x)  d M (x)$$
exist and do not depend on the choice of $\av$, provided $\nu \ast \av=
1$.
\end{Lemma}
\begin{proof} The functions $N_j$ and $P_j$  only depend on the
isomorphism  class of a finite ball around the corresponding points.
Hence, they are clearly measurable by the definition of a measurable graph. Moreover, they are invariant
under the action of the groupoid as this action results in graph
isomorphisms and hence preserves the isomorphism classes of finite
balls around the corresponding points. Thus, by Proposition
\ref{Generating} the functions $\xi_{N_j}$ and $\xi_{P_j}$ are
random variables.
 The remaining statements follow from the considerations
 concerning the integration of random variables in the
previous section and specifically from Definition
\ref{Integral-random-variable}.
\end{proof}

\begin{Definition}[The Zeta function of a measure graph]
 Let $(G,M)$ be a measure graph over a groupoid.  The complex function
$$Z(u) := Z_{(G,M)} (u) := \exp\left(\sum_{j\geq 1}
\frac{\overline{N}_j}{j} u^j\right)$$
is called \textit{Ihara Zeta function} of $(G,M)$. The numbers
$\overline{N}_j$, $j\in\NN$, appearing in its definition are called
the \textit{coefficients} of the Zeta function.
\end{Definition}

\begin{Proposition}\label{Eulerproduct}
Let $(G,M)$ be a measure graph over a groupoid. $Z = Z_{(G,M)}$ is a
holomorphic function on the disc $B_{(D-1)^{-1}}:= \{u \in \CC\,:\, |u| <
(D-1)^{-1}\},$ where $D$ is the maximal vertex degree of $G$.
Furthermore, $Z$ has an Euler product representation, i.e., for
each $u\in B_{(D-1)^{-1}}$ the equality
$$Z(u) = \prod_{j\geq 1} (1-u^j) ^{-\frac{\overline{P}_j}{j}}$$
holds.

\end{Proposition}
\begin{proof}
 The first assertion is a direct consequence of the observation $N_j(x)\leq
D(D-1)^{j-1}$ and the finiteness of the measure  $\av {\rm d}M$.  For the second statement let us note that for each $j$ and each $x \in V$ the equality
$$N_j(x) = \sum_{i\,|\, j} P_i(x) $$
holds, where $i\, | \, j$ means that $i$ divides $j$. Thus, using an integrated version of this identity we obtain
\begin{align*}
 \sum_{j\geq 1} \frac{\overline{N}_j}{j}u^j &=  \sum_{j\geq 1} \sum_{i\,|\, j} \overline{P}_i \frac{u^j}{j} \\
 \text{(interchanging summation)}&= \sum_{i \geq 1} \sum_{ j \geq 1} \overline{P}_i \frac{u^{ij}}{ij}\\
\text{(logarithmic series)} &= \sum_{i \geq 1} -\log(1-u^i)\frac{\overline{P}_i}{i}.
\end{align*}
Note that for $|u|<(D-1)^{-1}$ all the above series are absolutely convergent. Taking exponentials yields the claim.
\end{proof}
\begin{Remark}
We  emphasize that the validity of the Euler product representation
of $Z$ only owes to the validity of
 $$N_j(x) = \sum_{i\,|\, j} P_i(x).$$
\end{Remark}

\subsection{Examples} Here we discuss the Zeta functions of the measure graphs which were introduced in Subsection~\ref{Examples}. It is quite instructive to compute the quantities $\overline{N}_j$ and  $\overline{P}_j$ for these examples.

For finite and periodic graphs the exponent $\overline{P}_j$ in the Euler product representation can be further resolved and the product then be taken over certain geometric quantities. For this purpose we recall the following graph theoretic notions.

\begin{Definition}[Cycles]
\hspace{10cm}
\begin{itemize}
\item[(a)] Two closed paths $(e_1, \dots, e_m)$ and $(f_1, \dots, f_m)$ are said to be
equivalent if there is some $l$ such that $e_{i} = f_{i+l}$ for all $1
 \leq i\leq m$. Here we use the convention $f_{i} = f_{i+m}$. The corresponding
equivalence classes of closed paths are called {\em cycles}. The equivalence
 class of a closed path $C$ is denoted by $[C]$.
\item[(b)]  Cycles with their representatives being reduced and primitive are
called \emph{prime cycles}. The set of all prime cycles is denoted by $\mathcal{P}$.
\end{itemize}
\end{Definition}

\subsubsection{Finite graphs} Let $G= (V,E)$ be finite. As discussed at the end of the previous section $G$ is a measure graph over the trivial groupoid. The measure $M$ can be chosen to be the counting measure on $V$. Then, the constant function $1_V$ is an averaging function. In this situation we obtain
$$\overline{N}_j =\sum_{x\in V} N_j (x) \text{ and } \overline{P}_j =\sum_{x\in V} P_j (x),$$
that is, $\overline{N}_j$ is the total number of reduced closed paths of
length $j$ and $\overline{P}_j$ is the total number of primitive reduced
closed paths of length $j$.

\begin{Proposition} \label{finiteIhara}
Let $G=(V,E)$ be a finite graph and let $M$ be the counting measure. Then the Zeta function satisfies
$$
{Z}_{(G,M)}(u)= \prod_{[P] \in \mathcal{P}} \Big( 1- u^{\ell(P)} \Big)^{-1},
$$
where $\mathcal{P}$ is the set of prime cycles and $\ell(P)$ is the
length of $[P]$.
\end{Proposition}

\begin{proof}
Since $M$ is the counting measure on $V$ the quantity
$\overline{P}_j$ coincides with the number of primitive reduced
closed paths of length $j$.  Thus, we have $\overline{P}_j = j \cdot
 |\mathcal{P}_j|,$ with $\mathcal{P}_j$ being the set of prime
cycles of length $j$. Using this fact and reordering the product of
Proposition \ref{Eulerproduct} yields the claim.
\end{proof}

\begin{Remark}
For a finite graph the Ihara Zeta function is 'usually' defined via
the product of the previous proposition (see e.g. \cite{Terras}). Thus, for finite graphs our
definition of Zeta function coincides with the 'usual' one.
\end{Remark}

Our approach suggests to also consider the Zeta function coming from
the normalized measure giving the mass $1 /|V|$ to any vertex.
We call the arising Zeta function the \textit{normalized Zeta
function} of the finite graph  and denote it by $Z_{G,\,{\rm norm}}$.
Obviously $Z_{G,\,{\rm norm}}^{|V|} = Z_{(G,M)}$ holds. \label{Normalized-Zeta}

\subsubsection{Periodic graphs} \label{Periodic-Zeta} Let us first fix some notation. For a group $\Upsilon$ acting on some space $X$ we let $\Upsilon x = \{\gamma x \, : \, \gamma \in \Upsilon \}$ be the {\em orbit} and $\Upsilon_x = \{\gamma \in \Upsilon \, : \, \gamma x = x \}$ be the {\em stabilizer} of $x \in X$. Further, we denote by $X / \Upsilon = \{\Upsilon x\, : \, x \in X\}$ the collection of orbits.

\medskip

Let $(G,\Gamma)$ be a periodic graph (see Definition \ref{periodic graph}). Recall that in this case $M = m \circ \eta$ is the
counting measure on the vertices and $(G,M)$ is a measure graph over the groupoid $\calG = \Gamma$. In this situation we obtain
$$\overline{N}_j = \sum_{x\in F} \frac{N_j (x)}{|\Gamma_x|} \text{ and } \overline{P}_j = \sum_{x\in F} \frac{P_j (x)}{|\Gamma_x|}.$$
As these quantities (and thus $Z_{(G,M)}$) only depend on $G$ and $\Gamma$ we shall write $Z_{(G,\Gamma)}$ for the corresponding zeta function.

The group $\Gamma$ naturally acts on closed paths and on cycles of $G$. Namely, for a closed path $C = (e_1,\ldots,e_m)$ and $\gamma \in \Gamma$ we let $\gamma \, C = (\gamma e_1,\ldots ,\gamma e_m)$.  This action is compatible with passing from closed paths to cycles. Thus, for a cycle $[C]$ we can set $\gamma [C] = [\gamma\, C]$.

\begin{Proposition}
Let $(G,\Gamma)$ be a periodic graph. Its Zeta function satisfies
 $$Z_{(G,\Gamma)} (u) = \prod_{\Gamma [P] \in \mathcal{P}/\Gamma} \Big( 1- u^{\ell(P)} \Big)^{-\frac{1}{|\Gamma_{[P]}|}},$$
 where $\mathcal{P}$ is the set of all prime cycles and $\ell(P)$ is the length of $[P]$.
\end{Proposition}
\begin{proof}
We first note that the cardinality of the stabilizer $|\Gamma_{[P]}|$ and the length $\ell(P)$ are independent of the chosen representative of $\Gamma [P]$. Now the statement follows from Proposition~\ref{Eulerproduct} by reordering the product once we
show the equality
$$\frac{\overline{P}_j}{j} = \sum_{\{ \Gamma[P] \in \mathcal{P}/\Gamma\, : \, \ell(P) = j\}} \frac{1}{|\Gamma_{[P]}|}.$$
For any prime cycle $[P] \in \mathcal{P}$ the group $\Gamma_{[P]}$ acts on the set $[P]$. Hence, for $Q \in [P]$ The orbit-stabilizer formula yields
$$|\Gamma_{[P]}| = |\Gamma_{[P]} Q| |\Gamma_{Q}|,$$
where we use that the stabilizer of $Q$ of the $\Gamma_{[P]}$ action is $\Gamma_{Q}$. If we denote the set of all primitive reduced closed paths by  $\ow{\mathcal{P}}$, this leads to
\begin{align*}
 \sum_{\{ \Gamma  P \in \ow{\mathcal{P}}/\Gamma\, : \, \ell(P) = j\}} \frac{1}{|\Gamma_{P}|} &=   \sum_{\{ \Gamma  P \in \ow{\mathcal{P}}/\Gamma\, : \, \ell(P) = j\}} \frac {|\Gamma_{[P]} P|}{|\Gamma_{[P]}|}\\
 &= \sum_{\{ \Gamma[P] \in \mathcal{P}/\Gamma\, : \, \ell(P) = j\}} \sum_{\Gamma Q  \in \Gamma[P]} \frac {|\Gamma_{[Q]} Q|}{|\Gamma_{[Q]}|}\\
 &=\sum_{\{ \Gamma[P] \in \mathcal{P}/\Gamma\, : \, \ell(P) = j\}}  \frac{1}{|\Gamma_{[P]}|} \sum_{\Gamma Q  \in \Gamma[P]} {|\Gamma_{[Q]} Q|} .
\end{align*}
Let now $\{P_1,\ldots,P_l\}$ be a collection of Elements of $[P]$ of minimal cardinality such that $$\{\Gamma P_1,\ldots, \Gamma P_l\} = \Gamma [P].$$ It is readily verified that the sets $\Gamma_{[P]} P_1, \ldots, \Gamma_{[P]} P_l$ are a partition of $[P]$. Since $|[P]| = \ell(P) = j$ and $[P_k] = [P]$ for $1 \leq k \leq l$, we obtain
$$\sum_{\Gamma Q  \in \Gamma[P]} {|\Gamma_{[Q]} Q|}= \sum_{k=1}^l |\Gamma_{[P_k]} P_k| = \sum_{k=1}^l |\Gamma_{[P]} P_k|  = j.$$
Hence, in order to prove the desired identity for $\overline{P}_j/j$, it suffices to verify
$$\overline{P}_j = \sum_{\{ \Gamma  P \in \ow{\mathcal{P}}/\Gamma\, : \, \ell(P) = j\}} \frac{1}{|\Gamma_{P}|}.$$

To this end, fix a fundamental domain $F \subset V$ of finite co-volume.  For each orbit $\Gamma  P \in \ow{\mathcal{P}}/\Gamma$ there is a unique $x \in F$ such that there exist paths in  $\Gamma P$ starting in $x$. Thus, we say $\Gamma P$ starts in $x \in F$ if one of the paths in $\Gamma P$ starts in $x$. We obtain
$$\sum_{\{ \Gamma P \in \ow{\mathcal{P}}/\Gamma\, : \, \ell(P) = j\}} \frac{1}{|\Gamma_{P}|} = \sum_{x \in F} \sum_{\ato{\{ \Gamma P \in \ow{\mathcal{P}}/\Gamma\, : \, \ell(P) = j}{\text{ and } \Gamma P \text{ starts in } x\}}} \frac{1}{|\Gamma_{P}|}.$$
In order to sum over paths starting in $x \in F$ (as required to compute $\overline{P}_j$ ) instead of orbits of paths we need to count how many different paths in $\Gamma P$ actually do start in $x$.
Now if $\Gamma P$ starts at $x$, then $\gamma \,\Gamma P$ starts in $x$ as well if and only if $\gamma \in \Gamma_x$. Hence, the number of distinct paths starting at $x$ is given by
$|\{Q \in \Gamma P\, : \, Q \text{ starts in } x\}| = |\Gamma_x / \Gamma_P| = |\Gamma_x|/|\Gamma_P|.$  These considerations yield
\begin{align*}
 \sum_{x \in F} \sum_{\ato{\{ \Gamma P \in \ow{\mathcal{P}}/\Gamma\, : \, \ell(P) = j}{\text{ and } \Gamma P \text{ starts in } x\}}} \frac{1}{|\Gamma_{P}|} = \sum_{x \in F} \sum_{\ato{\{ \Gamma P \in \ow{\mathcal{P}}/\Gamma\, : \, \ell(P) = j}{\text{ and } \Gamma P \text{ starts in } x\}}} \frac{1}{|\Gamma_{P}|} \cdot \frac{|\Gamma_x|}{|\Gamma_{x}|} = \sum_{x \in F} \frac{P_j(x)}{|\Gamma_x|} = \overline{P}_j.
\end{align*}
This finishes the proof.
\end{proof}

\begin{Remark}
\begin{itemize}
 \item In  \cite{CMS01} periodic graphs were the first class of examples of infinite graphs for which Zeta functions were defined with the above product representation serving as a definition. Thus, our concept of Zeta function of a periodic graph coincides with the one previously studied.

 \item In \cite{CJK13} Zeta functions of vertex transitive infinite $q$-regular graphs were introduced. If the action of the automorphism group is free these coincide with our Zeta function for periodic graphs. If the action is not free they differ by some exponent.

 \item The previous proposition can be seen as an analogue of Theorem 2.2 (iii) of \cite{GIL08} which states that
 $$\prod_{\Gamma [P] \in \mathcal{P}/\Gamma} \Big( 1- u^{\ell(P)} \Big)^{-\frac{1}{|\Gamma_{[P]}|}} = \exp\left(\sum_{j \geq 1} \frac{N^\Gamma_j}{j}u^j  \right),$$
 for a sequence of numbers $N_j^\Gamma$ that depend on geometric quantities of the graph. The crucial new insight is that the coefficients $N^\Gamma_j$ equal $\overline{N}_j$ and can thus be expressed as integrals over functions on the vertices of the graph. Indeed, this is the main observation that lies at the heart of our paper.
\end{itemize}
\end{Remark}

\subsubsection{Percolation graphs}

Let $(G_{\Gamma, {\rm perc}},M_{\mathbb{P}})$ be a percolation graph as  in  Section~\ref{expample:percolation graphs}. We obtain
$$\overline{N}_j =  \int_\Theta N_j^{G_{\Gamma, {\rm perc}}}((e,\theta))\, {\rm d} \mathbb{P} = \int_\Theta N_j^{G_\theta}(e)\,  {\rm d} \mathbb{P}$$
and
$$\overline{P}_j =  \int_\Theta P_j^{G_{\Gamma, {\rm perc}}}((e,\theta))\,  {\rm d} \mathbb{P} = \int_\Theta P_j^{G_\theta}(e)\,  {\rm d} \mathbb{P}.$$
This shows that the Zeta function of the percolation graph counts the averaged number of reduced closed paths which start at the neutral element in the realization $G_\theta$.

\begin{Remark}
 To the best of our knowledge Zeta functions of percolation graphs have not been discussed in the literature before.
 This underscores the generality of the approach proposed in the present paper.
\end{Remark}

\subsubsection{Graphings}
Let $\mathcal{X}$ be a graphing as in
Subsection~\ref{example:graphings} and let
$(G_\mathcal{X},M_\mathcal{X})$ the corresponding measure graph. We
obtain
$$\overline{N}_j =  \int_X N_j^{G_\mathcal{X}}((x,x))\, {\rm d} \mu \text{ and } \overline{P}_j =  \int_X P_j^{G_\mathcal{X}}((x,x))\, {\rm d} \mu .$$
The graph $G_x = (\mathcal{J}x,E_x)$ with $E_x = \{(y,z) \in
\mathcal{J}x \times \mathcal{J}x\, : \,  y = \iota_j z \text{ for
some } 1\leq j \leq n \}$  is isomorphic to the induced graph on the
fiber $\pi^{-1}(\{x\}) \subseteq V_\mathcal{X}$. Therefore, the
above formulas simplify to
$$\overline{N}_j =  \int_X N_j^{G_x}(x)\, {\rm d} \mu \text{ and } \overline{P}_j =  \int_X P_j^{G_x}(x)\, {\rm d} \mu.$$
For the corresponding Zeta function we shall simply write
$Z_\mathcal{X}$ instead of $Z_{(G_\mathcal{X},M_\mathcal{X})}$.

\begin{Remark}
 Zeta functions of graphings have occurred implicitly in the literature as the limit of Zeta functions of finite graphs.  We will discuss this  in detail in Subsection~\ref{subsection:weakly convergent graphs sequences}.
\end{Remark}

\section{A determinant formula for the Zeta function}
\label{Calculating-vertices} In this section we provide a
determinant formula for the Ihara Zeta function in terms of the
adjacency and the degree operator. The determinant formula states
that in a small neighborhood around $0$ the reciprocal of the
function $Z_{(G,M)}$ can be expressed as a determinant  (up to some
multiplicative factor). In particular, $Z_{(G,M)}$ must then be a
non-vanishing holomorphic function in a small neighborhood of $0$.
 For the proof we will follow the lines of
\cite{GIL09} and give details only when they need to be adapted to
our situation. For the case of finite graphs the presented approach
towards proving the determinant formula was first established in
\cite{StarkTerras-96}. However, before providing the determinant
formula we need to set the stage and introduce the operators and the
determinant. To this end, we discuss the von Neumann algebra
$\mathcal{N}(\mathcal{G},G)$  of a measure graph and introduce a
trace on it.

\subsection{The von Neumann algebra}
Let a measure graph $(G,M)$ over a groupoid $\calG$ be given. Then,
the graph can be seen as a bundle over $\Omega=\calG^{(0)}$ via $\pi
: V\to \Omega$. This induces a bundle structure on $L^2(G, M)$ and
this will be our point of view in the subsequent considerations.
More specifically, by Fubini's Theorem we can associate with $f\in
L^2(V, m\circ \eta)$  a family $(f_\omega)_{\omega \in \Omega}$ of
functions
\begin{equation}\label{Faserf}
f_\omega \in \ell^2(V^\omega, \eta^\omega)  = \ell^2(V^\omega) \;\: \mbox{such that}\;\:
f(x) = f_{\pi(x)}(x)
\end{equation}
for $m$-almost all $\omega \in \Omega$.

\smallskip

\begin{Remark}
On the technical level, our approach  can be   expressed via  direct
integral theory and the fact that there is a canonical isomorphism
\begin{equation}\label{iso}
L^2(V, M)\simeq \int_\Omega^{\oplus} \ell^2(V^\omega, \eta^\omega)
\, dm (\omega).
\end{equation}
Here, the field $\omega \mapsto \ell^2 (V^\omega,\eta^\omega)$ of
Hilbert spaces  is equipped with the canonical measurable structure
 induced by all measurable $f : \mathcal{G}\longrightarrow
[0,\infty]$ with $\eta^\omega (f^2) <\infty$ for all $\omega \in
\varOmega$. In our presentation we will not need to explicitly
invoke  direct integral theory and for this reason we do not
elaborate further on this. Instead we just recall (or rather adapt)
the corresponding considerations of Section 3 of  \cite{LPV} to our
context. The work \cite{LPV} in turn is a specialization of
corresponding considerations in \cite{Connes-1979}.  We refer to
these works for further discussion of the background and the
technicalities. We also refer the reader to e.g.
~\cite{Dixmier-1981} for background in direct integral theory.
\end{Remark}

\medskip

A special role will be played by those operators which respect the
bundle structure of $ L^2(V,M)$ given by \eqref{Faserf}. Namely, we
say that the (not necessarily bounded) operator $A:
L^2(V,M)\to L^2(V, M)$ is {\it decomposable} if there
exist operators $A_\omega : \ell^2(V^\omega,
\eta^\omega)\to \ell^2(V^\omega, \eta^\omega)$ such that
$(A f)_\omega (x) = (A_\omega f_\omega) (x)$, for almost every $\omega \in
\Omega$. We then write $A=\int_\Omega^\oplus A_\omega \,d m
(\omega)$.

\medskip

For $g\in \calG$, let the unitary operator $U_g$ be given by
\begin{equation*}
U_g \colon \ell^2(V^{s(g)},\eta^{s(g)}) \to
\ell^2(V^{r(g)},\eta^{r(g)}),\; U_g f (x) := f( g^{-1}x).
\end{equation*}

A family $(A_\omega)_{\omega\in \Omega}$ of bounded operators
$A_\omega: \ell^2(V^\omega,\eta^\omega)\to \ell^2(V^\omega,\eta^\omega)$
is called a {\em bounded random operator} if it satisfies :
\begin{itemize}
\item $\omega\mapsto \langle g_\omega, A_\omega f_\omega\rangle$ is
 measurable for arbitrary $f,g\in L^2(V, M)$.
\item There exists a $C\geq 0$ with $\|A_\omega\|\leq C$ for almost
all $\omega \in \Omega$.
\item The equivariance condition $A_{r(g)} = U_g A_{s(g)} U_g^*$  is satisfied for all
$g\in \calG$.
\end{itemize}
Two bounded random operators $(A_\omega), (B_\omega)$ are called
equivalent, $(A_\omega)\sim (B_\omega)$ if $A_\omega=B_\omega$ for
$m$-almost every $\omega\in \Omega$.  Each equivalence class of
bounded random operators $(A_\omega)$ gives rise to a bounded
operator $A$ on $L^2(V,M)$ by $A f (x) := A_{\pi(x)} f_{\pi(x)}$.
This allows us to identify the class of $(A_\omega)$ with the
bounded operator $A$. The set of classes of
 bounded random operators is denoted by $\calN(\calG,G)$. It is
 obviously an algebra (under the usual pointwise addition and multiplication
on the level of representatives). Moreover, it carries a norm given
by
$$\|[(A_\omega)_{\omega\in \Omega}]\| :=\inf\{ C\geq 0 : \|A_\omega\| \leq C
\;\: \mbox{for $m$-almost every $\omega \in \Omega$} \}.$$
(It is not hard to see that this is indeed well-defined and a norm.)
The following is the main theorem on the structure of the space of
random operators (see \cite[Thm.~V.2]{Connes-1979})

\begin{Theorem} \label{Neum} The set $\calN(\calG,G)$ of classes of
 bounded random operators is a von Neumann algebra.
\end{Theorem}

\subsubsection{The canonical trace}
The admissible setting of \cite{LPV} always gives a canonical weight
on the von Neumann algebra in question. In this section we present
details for our situation and  show, in particular, that then  this
weight is actually a trace. Further details can be found in Section
4 of \cite{LPV}.

\medskip

Let a measure graph $(G,M)$ over a groupoid $\calG$ be given and
$\calN(\calG,G)$ be the associated von Neumann algebra.
 Let $\calN^+(\calG,G)$ denote the set of non-negative self-adjoint operators in $\calN(\calG,G)$. We will show
that every operator $A \in \calN^+(\calG,G)$ gives rise to a {\em
new random variable} $\beta_A$. Integrating this random
variable, we obtain a weight on $\calN(\calG,G)$. This weight will
be shown to be a trace.  We start by associating a transversal
function as well as a random variable with each element in
$\calN^+(\calG,G)$. For a nonnegative function $f $ on $\mathcal{G}^{\omega}$ we  will denote by $q_\omega(f)$ the function on $\pi^{-1}
(\omega)$ given by
\begin{equation*}
q_\omega (f)(x) = \int_{\calG^{\omega}} \av(g^{-1}x) f(g)
d\nu^{\omega}(g)
\end{equation*}
for all $x \in V$. Thus, $q_\omega (f) $ is   the restriction of
$q(f)$ as defined at the end of  Section \ref{subsection: graph over
groupoid} to the fiber $V^\omega$.

To proceed  we will need the (usual) trace $\tr$ on the set
of non-negative operators on an Hilbert space. This trace is defined
by
$$\tr (A) :=\sum_{\iota\in I} \langle A e_\iota, e_\iota\rangle,$$
where $e_\iota$, $\iota\in I$, is an orthonormal basis. It does not
depend on the choice of the orthonormal basis. In particular, is is
invariant under conjugation with unitaries i.e. $\tr (U A U^*) = \tr
(A)$ for any unitary $U$. Note that this trace may well take the
value $\infty$.  We will need this trace for  the Hilbert space
$\ell^2 (V^\omega,\eta^\omega)$ for each $\omega \in\Omega$.
Suppressing the dependence on $\omega\in \Omega$, we will still use
the notation $\tr$.  For a function $f$ on $V^\omega$ we denote by
$M_f$ the operator of multiplication by $f$ in $\ell^2 (V^\omega,
\eta^\omega)$. Similarly, for  a function $f\in \calF^+ (V)$ we
denote by $M_{f_\omega}$ the operator of multiplication by the
restriction of $f$ to $V_\omega$.

We will need to take the trace of operators of the form $B M_f B$
for nonnegative measurable functions $f$ and bounded nonnegative
operators $B$. To do so  we note that - under the additional
assumption of boundedness of $f$ -  the operator $C :=
M_{\sqrt{f}} B $  is bounded with adjoint given by $C^* = B
M_{\sqrt{f}}$. Hence, we have for such $f$ that $B M_f B = C^* C$ is
nonnegative and we can take the trace. For general nonnegative  $f$
we define
$$\tr (B M_f B) := \lim_{n\to\infty} \tr (B M_{f_n} B)$$
with $f_n :=\max\{f,n\}$.

\begin{Lemma} \label{trarand}
Let $A\in \calN^+(\calG,G)$ be given.\\
(a) Then $\phi_A$, given by $\phi_A^\omega(f) :=\tr ( A_\omega^{1/2}
M_{q_\omega (f)} A_\omega^{1/2}),$ $f\in \calF^+( \calG)$, defines a
transversal function.\\
(b) Then $\xi_A$, given by $\xi_A^\omega(f) := \tr (A_\omega^{1/2}
M_{f,\omega} A_\omega^{1/2}),$ $f\in \calF^+ (V)$, is a random
variable.
\end{Lemma}
\begin{proof} These statements are contained in Lemma 4.2 of \cite{LPV}.
  For convenience of the reader we indicate the
proof of (b). The proof of (a) can be  done along very similar
lines:
 Clearly the map $f\mapsto \xi_A (f)$ is
monotone (i.e. $\xi_A (g) \leq \xi_A (f)$ for $g\leq f$). Hence, it
is  a measure. Moreover, this measure is supported on $V^\omega$ by
the very definition of $M_{f,\omega}$. The invariance of $\xi_A$
follows now easily from the invariance properties of the
$(A_\omega)$ and the invariance of the trace under conjugation with
a unitary operator.
\end{proof}

Let us recall the following definitions. A {\em weight} on a von
Neumann algebra $\calN$ is a map $\tau: \calN^+ \to [0, \infty ]$
satisfying $\tau(A + B)= \tau (A) + \tau(B)$ and $\tau(\lambda
A)=\lambda \tau(A)$ for arbitrary $A,B\in \calN^+$ and $\lambda \ge
0$. The weight is called {\em normal} if $\tau(A_n)$ converges to
$\tau(A)$ whenever $A_n$ is an increasing sequence of operators
(i.e. $A_n \leq A_{n+1}, n\in \NN$)  converging strongly to $A$. It
is called {\em faithful} if $\tau(A)=0$ implies $A=0$. It is called
{\em semifinite} if $\tau(A)=\sup\{\tau(B): B\leq A,
\tau(B)<\infty\}$.  If a weight $\tau$ satisfies $\tau(C
C^*)=\tau(C^* C)$ for arbitrary $C\in \calN$ (or equivalently
$\tau(U A U^*)=\tau (A) $ for arbitrary $A\in \calN^+$ and unitary
$U\in \calN$, cf.~ \cite[Cor. 1 in I.6.1]{Dixmier-1981}), it is
called a {\em trace}.

\medskip

In our situation we have a canonical candidate for a weight at our
disposal. This is introduced in the next proposition.

\begin{Proposition}[Introducing $\tau$]
 Let $(G,M)$ be a measure graph over $\calG$.  For $A\in \calN^+(\calG,G)$, the
expression
\begin{equation*}
\tau(A) :=   \int_\Omega \tr (A_\omega^{\frac{1}{2}} M_{\av_\omega}
A_\omega^{\frac{1}{2}} ) d m (\omega)
\in[0,\infty]
\end{equation*}
does not depend on $\av \in \calF^+(G)$ provided $\nu \ast \av\equiv 1$.
\end{Proposition}
\begin{proof}
This follows directly  from Lemma \ref{ind} and Lemma \ref{trarand}.
\end{proof}

We are going to show that the map
$$\tau : \calN^+(\calG,G) \to [0,\infty], A\mapsto \tau (A),$$ is a
faithful, semifinite normal trace on $\calN(\calG,G)$. In order to
do so we will present some additional pieces of information. These
may also be interesting in  their own right.

\medskip

An operator $K$ on $L^2(V,M)$ is called a {\em Carleman operator}
(cf.~\cite{Weidmann-1980} for further details) if there exists a $k
\in \calF(V^{(2)})$ with
\begin{equation*} k(x,\cdot)
\in \ell^2(V^{\pi(x)},\eta^{\pi(x)}) \;\:\mbox{for all $x \in V$}
\end{equation*}
 such that for any $f\in L^2 (V,M)$
\begin{equation*}
Kf(x) = \int_{V^{\pi(x)}} k(x,y) f(y)
d\eta^{\pi(x)}(y) = \sum_{y \in V^{\pi(x)}}  k(x,y) f(y) =:K_{\pi(x)}f_{\pi(x)}(x)
\end{equation*}
in the sense of $L^2 (V,M)$.
 This $k$ is called the {\em
kernel} of $K$: Obviously, $K = \int_\Omega^\oplus K_\omega dm$.
Let $\calK$ be the set of all Carleman operators satisfying  for all
$g\in \calG$
\begin{equation}
\label{carl} k(gx,gy) = k(x,y)
\end{equation}
for all $x,y\in V^{(2)}$.

For a Carleman operator $K$  the expressions $\tau(K K^*)$ and
$\tau(K^* K)$ can directly be calculated.

\begin{Proposition}[Proposition 4.5 of \cite{LPV}] \label{tausch}
Let $(G,M)$ be a measure graph over $\calG$. Let $K \in \calK$ with
kernel $k$  be given. Then we have
\begin{equation*} \tau(K^* K) = \int_\Omega \int_{V^\omega} \int_{V^\omega}
\av(y) \vert k(x,y) \vert^2 d\eta^\omega(x) d\eta^\omega(y) d m
(\omega) = \tau(K K^*), \end{equation*} for any $\av \in \calF^+(V)$
satisfying $\nu \ast \av \equiv 1$.
\end{Proposition}

In our situation all elements of the von Neumann algebra are
Carleman operators.

\begin{Proposition}[$\calN(\calG,G) = \calK$]
 Let $(G,M)$ be a graph over $\calG$. Then, any element of
 $\calN(\calG,G)$ is a Carleman operator.
\end{Proposition}
\begin{proof} As shown in Proposition 4.4 of  \cite{LPV} the set $\calK$ is a right
ideal in $\calN(\calG,G)$. Now, obviously, the identity is a
Carleman operator (with kernel given by $k(x,y) = 1$ if $x = y$ and
$k(x,y) = 0$ otherwise). Hence, the desired statement follows.
\end{proof}

After these preparations we can now state (and prove) the main
result of this subsection.

\begin{Theorem} \label{snw} Let $(G,M)$ be a measure graph over
$\calG$. Then, the map
$$\tau :\calN^+(\calG,G) \to [0,\infty]$$
is a faithful, semifinite, normal trace. Furthermore, this trace is
finite, i.e., $\tau$ gives a finite value when applied to the
identity. Thus, $\tau$ can be uniquely extended to a continuous
linear  map on all of $\calN(\calG,G)$.
\end{Theorem}
\begin{proof} It follows from \cite{Connes-1979} (see Theorem 4.2 of  \cite{LPV} as
well) that $\tau$ is a faithful, semifinite weight. From the
preceding two propositions it follows immediately that $\tau$ is a
trace.  Now, a direct calculation shows that for the identity $I$ we
obtain
$$\tau (I) = \int_V \av(x) dM (x).$$
By definition of a measure graph the  above integral is finite and
the identity has a finite trace. Whenever $A$ is now a nonnegative element in the von
Neuman algebra we  have $A \leq \|A\| I$ (where $\|A\|$ denotes the
norm of $A$). Thus, monotonicity of the trace gives
$$0\leq \tau (A) \leq \|A\|\tau (I).$$
So, the trace is a continuous functional  on the nonnegative
elements. Given this, existence and uniqueness of a continuous
extension to the whole von Neuman algebra is straightforward.
\end{proof}

We finish this section by introducing  two special Carleman
operators.

\begin{Proposition}[The adjacency and the degree operator] Let $a_G$ be the
adjacency matrix of the graph $G$ and $\deg$ its degree function. Then $a_G$ is
the kernel of a Carleman operator which belongs to $\mathcal{N}(\mathcal{G},G)$. The associated  operator will be denoted by
$A_G$ and called the adjacency operator. Furthermore, multiplication by $\deg$
provides a Carleman operator belonging to $\mathcal{N}(\mathcal{G},G)$ which will be denoted by ${\rm Deg}_G$ and called
the degree operator.
\end{Proposition}
\begin{proof} We first show the measurability statements of the corresponding kernels. The function $a_G$ is measurable on $V^{(2)}$ by definition of a measurable graph. The kernel of ${\rm Deg}_G$ is given by
$$V^{(2)} \to \RR,\, (x,y) \mapsto 1_{\{(z,z)\, :\, z \in V\}}(x,y) \deg (y). $$
The measurability of this function follows from the Propositions \ref{Measurability-basic-quantities} and \ref{Measurability-product} and the definition of the $\sigma$-algebra on $V^{(2)}$. The invariance of both kernels under the action of $\mathcal{G}$ is clear as $\mathcal{G}$ acts upon $V$ by graph isomorphisms.
\end{proof}

\subsection{The determinant formula}

In order to state the main theorem of this section we need a holomorphic determinant on
certain operators of $\mathcal{N}(\mathcal{G},G)$. A natural way to obtain such
a functional is to set
$${\rm det}_\tau(T) := \exp\, \circ \, \tau \, \circ \log
(T).$$
Here, the logarithm of an operator $T$ with $\|I-T\| < 1$ is defined via the
power series expansion of the main branch of the logarithm around 1. Namely, we
let
$$\log(T) = - \sum_{k=1}^\infty \frac{1}{k}(I-T)^k. $$
Since the trace $\tau$ is continuous in the norm topology  the determinant function
$${\rm det}_\tau: \{T \in\mathcal{N}(\mathcal{G},G)\,:\, \|T-I\|< 1\} \to \CC,\, T \mapsto {\rm det}_\tau(T)  $$
is holomorphic.
\begin{Remark}
The above definition of a determinant functional suffices for the purposes of this paper. As discussed in \cite{GIL08} one can use holomorphic functional calculus to extend  the definition of the determinant to operators whose convex hull of the spectrum does not contain $0$.
\end{Remark}

\begin{Definition}[$L^2$-Euler characteristic]
 Let $(G,M)$ be a measure graph over a groupoid. We set $Q_G = {\rm Deg}_G -
I$. The number
 $$\chi_{(G,M)} = \frac{1}{2}\tau(I-Q_G)$$
 is called the $L^2$-Euler characteristic of $(G,M)$.
\end{Definition}

\begin{Theorem}[Determinant formula] \label{Determinant formula}
 Let $(G,M)$ be a measure graph over a groupoid. Then, for $|u| < R^{-1}$
the Zeta function satisfies the equation
$$Z_{(G,M)}(u) ^{-1} = (1-u^2)^{-\chi_{(G,M)}} {\rm det}_\tau (I - u A_G +
u^2 Q_G),$$
where  $R :=  \frac{D + \sqrt{D^2 + 4 D}}{2}$ and $D$ is the maximal vertex degree of $G$.
\end{Theorem}

The rest of this section is devoted to proving the theorem. As
mentioned earlier, we closely follow \cite{GIL09} and omit some
calculations that only require algebraic manipulations of identities
of bounded operators.

We define a series of Carleman operators $A_j$ recursively via $A_0 := I$,
$A_1 := A_G$, $A_2 := A_G^2 - Q_G - I$ and $A_j := A_{j-1} A_G - A_{j-2} Q_G$, for
$j \geq 3$. Let $a_j$ denote the corresponding kernels. The following lemma
explains the importance of the $A_j$'s.
\begin{Lemma} \label{Determinant Lemma 1}
For each $p,q \in V$ the kernel entry $a_j(p,q)$ is equal to the number of
proper paths of length $j$ starting in $p$ and ending in $q$. Furthermore, for each $j$ the norms of the corresponding operators satisfy $\|A_j\| \leq R^j$, where $R :=  \frac{D + \sqrt{D^2 + 4 D}}{2}$ and $D$ is the maximal vertex degree of $G$.

\end{Lemma}
\begin{proof}
Note that multiplying Carleman operators is essentially done via matrix multiplication with the corresponding kernels. Thus, the first statement can be deduced as in the proof of Lemma 5.2 in \cite{GIL09}.  For the statement on the norm we note that the operator $A_G$ is given by the family $(A_{G,\,\omega})$, where $A_{G,\,\omega}$ is the adjacency operator on the fiber $V^\omega$. Thus, for $f =(f_\omega) \in L^2(V,M)$ we obtain
$$\|A_Gf\|^2 = \int_{\Omega} \|A_{G,\omega} f_\omega\|^2 dm(\omega) \leq D^2 \int_{\Omega} \|f_\omega\|^2 dm(\omega) = D^2 \|f\|^2,$$
showing $\|A_1\| = \|A_G\| \leq D.$ As $Q_G+ I$ is multiplication by $\deg$ we obtain $\|A_2\| \leq D^2 + D$. The rest follows by induction on $j$, using the inequality $\|A_j\| \leq D(\|A_{j-1}\| + \|A_{j-2}\|).$
\end{proof}

The functions $a_j$ count the number of proper paths. As such paths may
still have tails we need to introduce some further quantity.
Following the notation in \cite{GIL09} we let $t_j(x)$ be the number
of proper closed paths of length $j$ starting in $x$ which have a
tail. By the definition of a measure graph the function $V \ni x
\mapsto t_j(x)$ is measurable and we set
$$t_j := \int_V t_j(x) \av(x)dM(x).$$
\begin{Lemma} \label{Determinant Lemma 2}
 \begin{itemize}
 \item[(a)] For each $j \geq 0$ we have $\overline{N}_j = \tau(A_j) - t_j$.
  \item[(b)] $t_1 = t_2 = 0$, and, for $j\geq 3$, $t_j = t_{j-2} +
\tau((Q_G -I)A_{j-2})$.
  \item[(c)] For each $j \geq 0$ the equality $t_j = \tau \left( (Q_G -I)
\sum_{i = 1}^{\left[ \frac{j-1}{2} \right] } A_{j-2i} \right)$ holds.
 \end{itemize}
\end{Lemma}

\begin{proof}
 By Lemma \ref{Determinant Lemma 1} we have $N_j(x) = a_j(x,x) - t_j(x)$.
Integrating with respect to the measure $\av dM$ yields statement (a). Next we turn to proving statement (b). Whenever $(x,y)$ is an edge we let $\ow{t}_j(x,y)$ be the number of reduced closed paths of length $j$ with tail starting with the edge $(x,y)$. If $(x,y)$ is not an edge we set $\ow{t}_j(x,y) = 0.$ The function $\ow{t}_j$ is the kernel of an operator belonging to $\mathcal{N}(\mathcal{G},G)$. Indeed, the required measurability of $\ow{t}_j$ follows from Proposition \ref{Measurability-basic-quantities} (c) and the invariance under the action of the groupoid is clear, as it acts by graph isomorphisms.  We obtain
\begin{align*}
  \int_V t_j(x) \av(x) dM(x) &= \int_V \av(x) \sum_{\{y\, : \, y \sim x\} } \ow{t}_j(x,y) dM (x)\\
  &=\int_{\Omega} \sum_{x \in V^\omega} \sum_{y\in V^\omega} \av(x) \ow{t}_j(x,y) dm(\omega)  \\
  &= \int_{\Omega} \sum_{y \in V^\omega} \sum_{x\in V^\omega} \av(y) \ow{t}_j(x,y) dm(\omega)\\
  &= \int_V \av(y) \sum_{\{x\, : \, x \sim y\} } \ow{t}_j(x,y) dM (y),
\end{align*}
where the third equality is due to the invariance of $\ow{t}_j(x,y)$ under the action of the groupoid and the invariance of $m$ (use Proposition \ref{tausch} with the kernel $k(x,y) = \ow{t}_j(x,y)^{1/2}$). Now (b) can be obtained from the equation
$$\sum_{\{x\, : \, x \sim y\}} \ow{t}_j(x,y) = t_{j-2}(y) + (\deg(y) - 2) a_{j-2}(y,y),$$
which is a consequence of the following observation: Each closed of
length $j$ with tail can be shortened to a closed path of length
$j-2$ by removing its first and last edge. Thus, we have to count
the number of ways in which one can extend reduced closed paths of
length ${j-2}$ starting in $y$ to reduced closed paths with tail
starting in neighbors of $y$. There are two possibilities: Either
such a path of length $j-2$ has a tail or not. There are exactly
$t_{j-2}(y)$ reduced closed paths with tail starting in $y$. If one
wants to avoid backtracking, each of these can be extended in
$\deg(y) -1$ ways to neighbors of $y$. There are $a_{j-2}(y,y) -
t_{j-2}(y)$ reduced closed paths without tails. To avoid
backtracking there are $\deg(y)-2$ possibilities to extend each of
them. Thus, we obtain
$$\sum_{\{x\, : \, x \sim y\}} \ow{t}_j(x,y) = (\deg(y) - 1)t_{j-2}(y) + (\deg(y) - 2) (a_{j-2}(y,y)-t_{j-2}(y))$$
and the claim follows after integrating these identities. Assertion (c) is an immediate consequence of (b).
\end{proof}

\begin{Lemma} \label{Determinant Lemma 3}
 For $j\geq 0$ let $B_j := A_j - (Q_G - I)\sum_{i =
1}^{\left[j/2\right]}A_{j-2i}$. Then
\begin{itemize}
\item[(a)] $$\tau(B_j) = \begin{cases}\overline{N}_j - \tau(Q_G -
I),\, &j \text{ even}\\ \overline{N}_j,\, &j \text{ odd.}\end{cases}$$
\item[(b)] $$\tau\left(\sum_{j\geq 1} B_j u^j\right) = \tau \left(-u
\frac{d}{du} \log (I - A_G u + Q_G u^2) \right),\, |u| < R^{-1}.$$
\end{itemize}
\end{Lemma}
\begin{proof}
The algebraic manipulations required to prove (a) are contained in the proof of Lemma 7.2
of \cite{GIL09}. To carry them out in our setting one only needs the identities of Lemma
\ref{Determinant Lemma 2}. Assertion (b) is given by Corollary 7.4 of
\cite{GIL09} whose proof involves algebraic manipulations of power series that require the estimate on the norm of the $A_j$ of Lemma \ref{Determinant Lemma 1}.
\end{proof}

\begin{proof}[Proof of Theorem \ref{Determinant formula}]
 Since $\tau$ is norm continuous we obtain  by Lemma \ref{Determinant Lemma 3}
(a) that the equality
 \begin{align*}
 \tau\left(\sum_{j\geq 1} B_j u^j\right) &= \sum_{j\geq 1}\tau(B_j )
u^j = \sum_{j\geq 1} \overline{N}_j u^j - \sum_{j\geq 1} \tau(Q_G - I) u^{2j} \\
&= \sum_{j \geq 1} \overline{N}_j u^j - \tau(Q_G - I)\frac{u^2}{1-u^2}
 \end{align*}
 holds. This computation together with \ref{Determinant Lemma 3} (b) yields
 \begin{align*}
  u \frac{d}{du} &\log Z_{(G,M)} = \sum_{j\geq 1} \overline{N}_j u^j\\ =&
\tau \left(-u
\frac{d}{du} \log (I - A_G u + Q_G u^2) \right) -
\frac{u}{2}\frac{d}{du}\log(1-u^2) \tau(Q_G-I).
 \end{align*}
 Dividing by $u$ for $u \neq 0$, integrating from $0$ to $u$ and taking
exponentials yield the claim.
\end{proof}

\section{Zeta function  and  integrated density of states on essentially regular graphs}\label{essRegular}
In this section we discuss properties of the Ihara Zeta function for
essentially regular graphs. As for regular graphs in the finite
case, their Zeta function is closely related to the spectrum of the
adjacency operator.

\begin{Definition}[Essentially regular graph]
Let $(G,M)$ be a  measure graph over a groupoid. Then, $G$ is called
\textit{essentially $(r+1)$-regular} if $\deg(x) = r+1$ for
$M$-almost all $x \in V$.
\end{Definition}

Let $(G,M)$ be a measure graph over a groupoid $\mathcal{G}$. Recall
that each self-adjoint operator $T \in \mathcal{N}(\mathcal{G},G)$
possesses a spectral measure $\mu_T$ on the spectrum $\sigma(T)$ of $T$ which is
uniquely determined by the identity
$$\tau(\varphi(T)) = \int_{\sigma(T)}  \varphi(\lambda) d\mu_T(\lambda), \text{ for each } \varphi \in C(\sigma(T)).$$

\begin{Remark}
The measure $\mu_T$ is sometimes called the abstract integrated
density of states of $T$, see Chapters 5 and 6 of \cite{LPV} for a
detailed discussion.
\end{Remark}

We let $\mu_G:= \mu_{A_G}$ be the spectral measure of the adjacency
operator of $(G,M)$. The following is the main observation of this
paragraph. It connects the spectral theory of $A_G$ and the Zeta
function $Z_{(G,M)}$.

\begin{Theorem}\label{determinant-formula-regular}
Let $(G,M)$ be a measure graph over a groupoid $\mathcal{G}$. Assume
that $(G,M)$ is essentially $(r+1)$-regular. Then, for $|u| < R^{-1}$ its Zeta function
satisfies
 $$Z_{(G,M)}(u)^{-1} = (1-u^2)^\frac{(r-1)\tau(I)}{2} \exp\left( \int_{\sigma({A_G})} \log(1- u \lambda + u^2 r) d\mu_G(\lambda)  \right),$$
 where $2R = (r+1) + \sqrt{(r+1)^2 + 4 (r+1)}$ and $\log$ is the main branch
of the complex logarithm on the sliced plane $\CC_- := \CC \setminus \{x \in \RR\, : \, x \leq 0\}$.
\end{Theorem}
\begin{proof}
To prove the theorem we will compute the quantities involved in the
determinant formula of Theorem \ref{Determinant formula}.  As $G$ is
essentially $(r+1)$-regular, the operator $Q_G = {\rm Deg}_G - I$ is
multiplication by the constant $r$. Thus, we obtain
  $$\chi_{(G,M)} =\frac{1}{2} \tau (I-Q_G) = \frac{(1-r) \tau(I)}{2}.$$
Using the inclusion  $\sigma(A_G) \subseteq [-(r+1),r+1]$ and $r< R$
an elementary computation shows $1-\lambda u+u^2r\in\CC_-$, whenever
$|u|<R^{-1}$ and $\lambda \in \sigma(A_G)$. Thus, for all $|u|<
R^{-1}$ the function $\lambda \mapsto \psi(\lambda) :=  \log(1 -
u\lambda +ru^2)$ is continuous on $\sigma(A_G)$. We obtain
  $$\log(1 - uA_G + u^2 Q_G) = \log(1 - uA_G + u^2 r I) =  \psi (A_G).$$
  Hence, the definition of $\mu_G$ yields
  $$\tau(\log(1 - uA_G + u^2 Q_G)) = \int_{\sigma(A_G)} \psi(\lambda) d\mu_G(\lambda).$$
  Now the claim follows from Theorem \ref{Determinant formula}.
\end{proof}

\begin{Remark}
 In \cite{GZ} the previous theorem serves as a definition for the
Ihara Zeta function of certain infinite regular graphs. There, the
measure $\mu_G$ is given by the integrated density of states (the
KNS-spectral measure)  of a weakly convergent graph sequence. Furthermore, for
vertex-transitive (regular) graphs the above formula recovers Theorem 1.3 of \cite{CJK13} which was shown there by different means via an analysis of Bessel functions and heat kernels.  In both cases the authors
consider spectral measures associated with the Laplacian $L = {\rm
Deg}_G - A_G$ instead of measures corresponding to $A_G$. However,
for essentially regular graphs the operator ${\rm Deg}_G$ is a
constant multiple of the identity and, thus, the Laplacian and the
adjacency matrix are essentially the same operators up to a shift by
a constant.
\end{Remark}

\begin{Corollary}\label{corollary:holomorphic extension}
Let $(G,M)$ be a measure graph over a groupoid $\mathcal{G}$. Assume
that $(G,M)$ is essentially $(r+1)$-regular. Then its Zeta function can be continued to a holomorphic function on the open set
$$\mathcal{O} := \{z \in \CC \,  : \, |z| < r^{-1/2}\} \setminus \{x \in \RR\,:\, r^{-1} \leq |x| \leq 1\} .$$
\end{Corollary}

\begin{proof}
 Theorem \ref{determinant-formula-regular}, the inclusion $\sigma(A_G) \subseteq [-(r+1),r+1]$ and the elementary fact that $1-\lambda u + u^2 r \in \CC_-$ for each pair $(u,\lambda) \in \mathcal{O}\times [-(r+1),r+1]$ show the claim.
\end{proof}

\section{Convergence of measure graphs and of Zeta functions}\label{Result}
In this section we provide a general result on convergence of Zeta
functions. More specifically, we introduce the notion of weak
convergence of measure graphs and show that it implies local compact convergence
of their Zeta functions. Our notion of weak
convergence of measure graphs generalizes the concept of weak
convergence of sequences of finite graphs which has been introduced
in \cite{BS}. While not noted explicitly, particular cases of weak
convergence of finite graphs are at the core of all the earlier
attempts to provide a Zeta function for infinite graphs via
approximation \cite{CMS,GZ,GIL08,GIL09}. Therefore, our convergence
theorem generalizes and unifies the results of the mentioned
literature.

Moreover, if the measure graphs converge to a limit we can interpret
the limit of the Zeta functions as the Zeta function of the limit
measure graph. In this way, we obtain a continuity statement for the
Zeta function.



\bigskip

We start with the following elementary but general observation on the convergence of Zeta functions.

\begin{Proposition} \label{proposition:convergence of Zeta functions}
Let $((G_n,M_n))$ be a sequence of measure graphs with averaging
functions $(\av_n)$ and vertex sets $(V_n)$ and uniform vertex
degree bound $D\in \NN$. Assume that the sequence $\int_{V_n} \av_n
(x) dM_n (x)$   is bounded and the limit
$$F_j:= \lim_{n\to \infty} \int_{V_n} N^{G_n}_j (x) \av_n (x) {\rm d} M_n (x)  $$
exists for any $j\in \NN$. Then $$Z_{(G_n,M_n)} \to
Z, $$
uniformly on compact subsets of $B_{(D-1)^{-1}}:= \{u\in \CC\, : \, |u| < (D-1)^{-1}\},$ where $Z$ is given by
$$Z(u) = \exp \left(\sum_{j\geq 1}\frac{F_j}{j}u^j\right).$$
If $(G,M)$ is a measure graph with vertex set $V$ and averaging function $\av$ which satisfies
$$F_j = \int_{V} N^{G}_j (x) \av (x) {\rm d} M (x)$$
for any $j\in \NN,$ then $Z = Z_{(G,M)}$.
\end{Proposition}
\begin{proof} By assumption we have convergence of the coefficients
$$\overline{N}_j^{(n)}:=\int_{V_n} N^{G_n}_j (x) \av_n(x) {\rm d} M_n (x)  \to  F_j.$$
Since $(M_n(V_n))$ is bounded by some constant $C \geq 0$ and $N_j$ satisfies $|N_j(x)| \leq D(D-1)^{j-1}$ (cf. proof of Proposition~\ref{Eulerproduct}) we obtain $\overline{N}_j^{(n)} \leq C D(D-1)^{j-1}$.  This  easily gives the desired
convergence.
\end{proof}

\begin{Remark} Note that boundedness of the sequence $\int_{V_n} \av_n
(x) dM_n (x)$ follows e.g. from boundedness of the sequence
$(M_n(V_n))$.
\end{Remark}


The previous proposition shows that in order to establish
convergence of Zeta functions it suffices to verify pointwise
convergence of their coefficients. The following concept of
statistical convergence of local patterns will ensure that.  Recall
that for $\alpha \in \mathcal{A}^D$ and $\ow{V}\subseteq V$ the set
$\ow{V}_\alpha$ denotes the collection of all vertices in $\ow{V}$
whose $\rho(\alpha)$-ball is isomorphic to $\alpha$.

\begin{Definition}[Weakly convergent measure graph sequences]
Let $((G_n,M_n))$ be a sequence of measure graphs
with uniform vertex degree bound $D \in \NN$. Let $(V_n)$ be the corresponding vertex sets and let $(\av_n)$ be the corresponding averaging functions. We call $((G_n,M_n))$ {\em weakly
convergent} if for each  $\alpha \in \mathcal{A}^D$, the  \textit{
frequency of $\alpha$}
\[
p(\alpha) := \lim_{n\to\infty} \int_{V_{n,\alpha} }\av_n(x) {\rm d} M_n(x),
\]
exists. In this case, a measure graph $(G,M)$ with vertex set $V$ and averaging function $\av$ is called a {\em  weak limit} of $((G_n,M_n))$ if for each $\alpha \in\mathcal{A}^D$ we have
$$p(\alpha) = \int_{V_\alpha} \av(x) {\rm d} M(x).$$
\end{Definition}

\begin{Remark}\label{remark:weakly convergent sequences}
Let $G$ be a finite graph equipped with the normalized counting measure on its vertex set and let $\av$ the function constantly equal to 1. Then, for $\alpha \in \mathcal{A}^D$ we have
$$\int_{V_\alpha}\av {\rm d} M = \frac{|\{x \in V\, : \, \pi_{\rho(\alpha)}(x) = \alpha\}|}{|V|} =: p(G,\alpha).$$
Thus, when interpreted as measure graphs with normalized measure, a sequence of finite graphs $(G_n)$ converges weakly if and only if for any $\alpha \in \mathcal{A}^D$ the sequence $(p(G_n,\alpha))$ converges. Sequences of connected graphs which satisfy the latter condition were introduced in \cite{BS} and called weakly convergent. Therefore, the notion of weak convergence for measure graphs is a generalization. With this observation it is not hard to come up with examples. For instance, considering regular graphs, one obtains such sequences from sofic approximations of groups. Further details including a precise definition for soficity of groups are given in Section~\ref{Actions}. Other examples are discussed at the end of this section.
\end{Remark}

The next lemma shows that the coefficients of the Zeta function only
depend on local patterns. This is the reason why it is compatible
with weak convergence.  Recall that for a fixed $r \geq 0$ the set
$\Vfr$ is the collection of all vertices whose rooted connected
component has radius $r$. We have seen in
Proposition~\ref{proposition:decomposition of V} that these sets are
measurable and naturally appear in a decomposition of $V$ into local
patterns. For $\alpha = [(G,x)] \in \mathcal{A}^D$ we define
 $$N_j(\alpha) := N_j^G(x). $$
 This is independent of the particular choice of the representative of $\alpha$.

 \begin{Lemma} \label{lemma:computation of Nj}
  Let $(G,M)$ be a measure graph an let $j\in \NN$. For each natural $n \geq j/2 +  1$ the identity
  $$\overline{N}_j := \sum_{\alpha \in \mathcal{A}_n^D} N_j(\alpha) M(\alpha)  + \sum_{r = 0}^{n-1} \sum_{\alpha \in \mathcal{A}_r^D} N_j(\alpha)  M^{\rm fin}(\alpha) $$
  holds, where
  $$ M(\alpha) = \int_{V_\alpha} \av {\rm d}M$$
  and
  $$M^{\rm fin}(\alpha) = \int_{V^{{\rm fin}, \rho(\alpha)}_\alpha} \av {\rm d}M.$$
 \end{Lemma}

 \begin{proof}
   Proposition~\ref{proposition:decomposition of V} shows that for each $n\in  \NN$ the set $V$ can be disjointly written as
  $$ V= \bigcup _{\alpha \in \mathcal{A}_n^D} V_\alpha \cup \bigcup_{r = 0}^{n-1} \bigcup_{\alpha \in \mathcal{A}_n^D} \Vfr_\alpha. $$
A closed path of length $j$ starting in $x$ can at most reach the
distance $j/2 +  1$ from $x$. Therefore, the value $N_j(x)$ only
depends on the isomorphism class of the ball of radius $j/2 + 1$
around $x$. This implies that for $n \geq j/2 + 1$ the function
$N_j$ is constant on each of the sets appearing in the above
decomposition. In particular, for $\alpha \in \mathcal{A}_n^D$ it
equals $N_j(\alpha)$ on $V_\alpha$ and for $\alpha \in
\mathcal{A}_r^D$ with $0\leq r \leq n-1$ it equals $N_j(\alpha)$ on
$\Vfr_\alpha$. Splitting the integral which appears in the
definition of $\overline{N}_j$ according to this decomposition
yields the statement.
 \end{proof}

In the previous lemma there appeared a term depending on the measure
of local patterns in finite connected components. The following
lemma shows that these are preserved under weak convergence.

\begin{Lemma} \label{lemma:convergence of finite patterns}
 Let $((G_n,M_n))$ be a weakly convergent sequence of measure graphs with vertex sets $(V_n)$ and averaging functions $(\av_n)$.  For each $\alpha \in \mathcal{A}^D$ the limit
 $$p^{\rm fin}(\alpha):= \lim_{n\to\infty} \int_{ V_{n,\alpha}^{{\rm fin},\rho(\alpha)}} \av_n(x) {\rm d} M_n(x) $$
 exists. If $(G,M)$ with vertex set $V$ and averaging function $\av$ is a weak limit of $((G_n,M_n))$ then

$$ p^{\rm fin}(\alpha) = \int_{ V_\alpha^{{\rm fin},\rho(\alpha)} } \av(x) {\rm d} M(x).$$
\end{Lemma}
\begin{proof}
 As seen in the proof of Proposition~\ref{proposition:decomposition of V} for each $r\geq 0$ we have
$$ \Vfr = \bigcup_{\alpha \in \mathcal{A}^D_r} V_\alpha \setminus \left(\bigcup_{\beta \in \mathcal{A}^D_{r+1}} V_\beta  \right).$$
 For $\alpha \in \mathcal{A}^D$ with $\rho(\alpha) =  r$ this implies
 $$\Vfr_\alpha = V_\alpha \cap \Vfr = V_\alpha  \setminus \left(\bigcup_{\beta\in \mathcal{A}^D_{r+1}} V_\beta  \right).$$
 Since the sets $V_\beta$ are pairwise disjoint for different $\beta\in \mathcal{A}^D_{r+1}$ the statement follows from the weak convergence.
\end{proof}

\begin{Theorem}[Continuity of the Zeta function] \label{theorem:convergence of zeta functions}
 Let $((G_n,M_n))$ be a weakly convergent sequence of measure graphs with uniform vertex degree bound $D\in \NN$. Then there exists a holomorphic function $Z:B_{(D-1)^{-1}} \to \CC$ such that
$$Z_{(G_n,M_n)} \to Z,  $$
uniformly on compact subsets of $B_{(D-1)^{-1}}.$ The function $Z$ is given by
$$Z(u) = \exp \left(\sum_{j\geq 1}\frac{F_j}{j}u^j\right),$$
with
$$F_j =  \sum_{\alpha \in \mathcal{A}_j^D} N_j(\alpha) p(\alpha)  + \sum_{r = 0}^{j-1} \sum_{\alpha \in \mathcal{A}_r^D} N_j(\alpha)  p^{\rm fin}(\alpha).$$
In particular, if $(G,M)$ is a weak limit of $((G_n,M_n))$, then $Z = Z_{(G,M)}$.
\end{Theorem}
\begin{proof}
 Together with the definition of weak convergence and Lemma~\ref{lemma:convergence of finite patterns}, Lemma~\ref{lemma:computation of Nj} shows that the coefficients of $Z_{(G_n,M_n)}$ converge and that their limits are given by the $F_j$. If $\alpha$ is the class of a  graph with one vertex and no edges for any $n$ we have $V_{n,\alpha} = V_n$. Therefore,  weak convergence implies the boundedness of the sequence $(M_n(V_n))$. With these observations the statement follows from Proposition~\ref{proposition:convergence of Zeta functions}.
\end{proof}

\begin{Remark}
 The proof of the previous theorem uses convergence of the statistics of local patterns and that the Zeta functions only depend on these statistics. Therefore, similar statements hold true for objects associated with the graph which only depend on these quantities as well. One example is the  integrated density of states of the adjacency operator or other finite range operators in  $\mathcal{N}(\mathcal{G},G)$  (see Section~\ref{essRegular} for a definition of the integrated density of states). More precisely, weak convergence of measure graphs implies weak convergence of the integrated density of states. We refrain from giving details.
\end{Remark}

We finish this section with two applications of our result on continuity of the Zeta function.

\subsection{Weakly convergent graph sequences} \label{subsection:weakly convergent graphs sequences}

As already discussed in Remark~\ref{remark:weakly convergent
sequences} weakly convergent graph sequences yield weakly convergent
measure graphs. It was shown in \cite{Ele} that for each weakly
convergent sequence of connected finite graphs $(G_n)$ with $|V_n|
\to \infty$ there exists a {\em limit graphing} $\mathcal{X}$ such
that for any $\alpha \in \mathcal{A}^D$ we have
$$\lim_{n\to \infty} p(G_n,\alpha) =  \mu(X_\alpha) = \int_{V_{\mathcal{X},\alpha}} \av \, {\rm d}M_\mathcal{X}, $$
where $X_\alpha = \{x \in X\, : \, B_{\rho(\alpha)}^{G_\mathcal{X}}((x,x)) \in \alpha\}.$  This implies that the sequence of finite measure graphs $((G_n,M_n))$ with normalized counting measure converge weakly  to the measure graph $(G_{\mathcal{X}},M_\mathcal{X})$. Recall that the Zeta function of a finite graph $G$ with normalized counting measure is denoted by $Z_{G,{\rm norm}}$ and that the Zeta function of a graphing $\mathcal{X}$ is denoted by $Z_\mathcal{X}$. From Theorem~\ref{theorem:convergence of zeta functions} and the previous discussion we immediately obtain the following.

\begin{Theorem}\label{theorem:weakly convergent graphs}
Let $(G_n)$ be a  weakly convergent sequence of finite connected
graphs with uniform vertex degree bound $D \in \NN$. There exists a
holomorphic function $Z:B_{(D-1)^{-1}} \to \CC$ such that
 $$Z_{G_n, {\rm norm}} \to  Z $$
 uniformly on compact subsets of $B_{(D-1)^{-1}}$ and the equality
$$Z = Z_{\mathcal{X}}$$
 holds for  every  limit graphing $\mathcal{X}$
of $(G_n)$.
\end{Theorem}

\begin{Remark} \label{Remark:Approximation_via_trace}
The above theorem extends the approximation results in the
literature:   For amenable, periodic graphs (with a discrete,
amenable group acting freely and co-finitely as automorphisms), an
approximation theorem for the Ihara Zeta function has been shown in
\cite{GIL08}. In \cite{CMS, GZ}, the authors show the existence of
the Ihara Zeta function for certain infinite regular graphs via a
convergence statement along covering sequences $(G_n)$ of finite,
regular graphs.
One way to see the existence of the limit is to use a theorem
of Serre \cite{Serre} on equidistribution of eigenvalues of Markov operators.
The convergence along F{\o}lner type subgraphs of amenable,
self-similar graphs has been proven in \cite{GIL09}. In all the
mentioned papers, the graph sequences under consideration can easily
be seen to be  weakly convergent.
Thus, for  every approximating sequence given or constructed in the
works \cite{CMS, GZ, GIL08, GIL09}, Theorem~\ref{theorem:convergence of zeta functions} shows
the convergence towards the Zeta function associated to the limit
graphing of the sequence. By comparing coefficients, it is not hard
to see that in all those cases, this function coincides with the
Ihara Zeta function of the original (infinite) graph. (For the
periodic case, we verify this explicitly in the proof of
Corollary~\ref{cor:soficapprox}.)
Thus, the above theorem generalizes and unifies the
approximation results in the literature.
\end{Remark}

\subsection{Percolation graphs}

Let $G_{\Gamma, {\rm perc}}$ be the percolation graph of the group $\Gamma$ with respect to the set of generators $I = \{g_1,\ldots,g_l,g_1^{-1},\ldots, g_l^{-1}\}$. The set of realizations of the percolation $\Theta = \{0,1\}^\Gamma$ is a compact space in the product topology and the corresponding Borel $\sigma$-algebra coincides with the product $\sigma$-algebra. Therefore, we have a notion of weak convergence of measures at hand. In the next theorem we prove that it is compatible with weak convergence of the associated measure graphs.
We point out here that this also shows that the notion of convergence of measure graph sequences goes beyond weak convergence of finite graphs towards
graphings.

\begin{Theorem} \label{theorem:percolationconv}
 Let $\mathbb{P}$ and $\mathbb{P}_n, n \in \NN,$ be $\Gamma$-invariant probability measures on $\Theta$. If the $(\mathbb{P}_n)$ converge  weakly to $\mathbb{P}$ the sequence  $((G_{\Gamma, {\rm perc}},M_{\mathbb{P}_n}))$ converges weakly to  $(G_{\Gamma, {\rm perc}},M_{\mathbb{P}})$. In particular, $\mathbb{P}_n \to \mathbb{P}$ weakly implies
 $$Z_{(G_{\Gamma, {\rm perc}},M_{\mathbb{P}_n})} \to Z_{(G_{\Gamma, {\rm perc}},M_{\mathbb{P}})}$$
 uniformly on compact subsets of $B_{(2l-1)^{-1}}$.
\end{Theorem}
\begin{proof}
 The vertex degree of the percolation graph is at most $2l$. We have seen that for all $\Gamma$-invariant measures on $\Theta$ the function $\av = 1_{\{e\} \times \Theta}$ is an averaging function. For $\alpha \in \mathcal{A}^{2l}$  and $n \geq 1$ this implies
 $$\int_{V_\alpha} \av\, {\rm d}M_{\mathbb{P}_n} = \int_\Theta  1_{\{\theta\in \Theta\,:\, B_r^{G_\theta} (e) \in \alpha\}}\, {\rm d} \mathbb{P}_n.$$
The set ${\{\theta\in \Theta\,:\, B_r^{G_\theta} (e) \in \alpha\}}$
is a finite union of  cylinder sets in $\Theta$. Therefore, the
indicator function in the above integral is continuous with respect
to the product topology. From the weak convergence of the measures
we infer
 $$\int_\Theta  1_{\{\theta\in \Theta\,:\, B_r^{G_\theta} (e) \in \alpha\}}\, {\rm d} \mathbb{P}_n \to \int_\Theta  1_{\{\theta\in \Theta\,:\, B_r^{G_\theta} (e) \in \alpha\}}\, {\rm d} \mathbb{P} = \int_{V_\alpha} \av\, {\rm d}M_{\mathbb{P}}.
 $$
 This shows weak convergence of the measure graphs. The 'in particular' statement follows from Theorem~\ref{theorem:convergence of zeta functions}.
\end{proof}
A particular example for the previous theorem is  Bernoulli percolation. For $0 \leq p \leq 1$ we let the measure $\mathbb{P}_p$ on $\Theta$ be given by
$$\mathbb{P}_p = \bigotimes_{\gamma \in \Gamma} \left[p \delta_{\{0\}} + (1-p) \delta_{\{1\}}\right]. $$
At each vertex independently it deletes all the emerging edges with probability $p$ and keeps them with probability $1-p$. Clearly, the $\mathbb{P}_p$ are $\Gamma$-invariant and weakly continuous in the parameter $p$. We obtain %
$$Z_{(G_{\Gamma, {\rm perc}},M_{\mathbb{P}_p})} \to 1, \text{ as } p \to 1,$$
and
$$Z_{(G_{\Gamma, {\rm perc}},M_{\mathbb{P}_p})} \to Z_{(G_0,\Gamma)}, \text{ as } p \to 0,$$
uniformly on compact subsets of $B_{(2l-1)^{-1}}$. Here $Z_{(G_0,\Gamma)}$ is the Zeta function of the Cayley graph when viewed as a periodic graph with respect to $\Gamma$.

\section{Actions of sofic groups on graphs}\label{Actions}
In this section we will investigate approximation of Zeta functions
on graphs which are periodic with respect to  the action of a sofic
group. The key step of our investigation is the
construction of a weakly convergent sequence of finite graphs reflecting the local statistics of the (periodic) graph in  Theorem
\ref{thm:soficapprox}. This will then be combined with the result on
weak convergence from the previous section.
The assertion of Theorem~\ref{thm:soficapprox} generalizes the previous
results provided in~\cite{CMS,GIL08}
in two ways: it is valid for all sofic groups
(i.e.\@ not bound to amenable or residually finite groups), as well as for all proper actions on the graph with finite covolume (rather than only for free actions with finite fundamental domain).
 As the class of sofic groups is rather large we thus obtain a
quite general result.

\medskip

There are various equivalent definitions for soficity of a group.
For our purposes, it will be convenient to work with the concept
of {\em almost homomorphisms}. To give a precise definition, we need
some preparation. For $n \in \NN$, we denote by
$\operatorname{Sym}(n)$ the symmetric group over $\{1, \dots, n\}$
with unit element $\operatorname{Id}_n$. This group is naturally
endowed with the {\em normalized Hamming distance} $d_H$, defined as
\[
d_H(\sigma, \tau):= \frac{\#\Big\{ a \in \{1, \dots, n\} \,|\, \sigma(a) \neq
\tau(a) \Big\}}{n}
\]
for $\sigma, \tau \in \operatorname{Sym}(n)$. Note that $d_H$ is a metric on
$\operatorname{Sym}(n)$,
see e.g.\@ \cite{Pestov-08}.

\begin{Definition}[Sofic groups] \label{defi:soficgroups}
A group $\Gamma$ with unity $e$ is called {\em sofic} if for every finite set
$T \subseteq \Gamma$ and for each $\varepsilon > 0$, there exist $n \in \NN$,
as well as a mapping
\[
\sigma: T \to \operatorname{Sym}(n): s \mapsto \sigma_s
\]
such that
\begin{enumerate}[(i)]
\item if $s,t, st \in T$, then $d_H(\sigma_s \sigma_t, \sigma_{st}) <
\varepsilon$,
\item if $e \in T$, then $d_H(\sigma_e, \operatorname{Id}_n) < \varepsilon$,
\item if $s,t \in T$ with $s \neq t$, then $d_H(\sigma_s, \sigma_t) \geq
1-\varepsilon$.
\end{enumerate}
If for $T$ and $\varepsilon$, there is some map $\sigma$ satisfying (i), (ii) and~(iii), 
then we say that $\sigma$ is an {\em almost homomorphism} for $(T, \varepsilon)$.
\end{Definition}

\begin{Remark} Sofic groups
have been invented by Gromov, cf.\@ \cite{Gromov-99}. The name was
given by Weiss in \cite{Weiss-00},  where the author defined the
notion for finitely generated groups. The class of all sofic groups
is large. In fact, it is not known whether all groups satisfy this
property. Sofic groups have become a flourishing subject in various
fields of mathematics, such as geometric group theory, ergodic
theory and symbolic dynamics. A survey  can  be found in
\cite{Pestov-08}.
\end{Remark}

We are now in position to prove the main theorem of this section.
Recall that periodic graphs have been introduced in Definition \ref{periodic graph}.

\begin{Theorem} \label{thm:soficapprox}
Let $(G,\Gamma)$ be a periodic graph with a sofic group $\Gamma$. Assume further
that the vertex degree of $G$ is bounded by $D\in \NN$ and that
$F$ is a fundamental domain of finite co-volume. Then, there exists a weakly convergent
sequence  of finite graphs $(G_n)$ such that for all $\alpha \in
\mathcal{A}^D,$
 $$\lim_{n\to \infty} p(G_n,\alpha) = \frac{\sum_{f \in F_{\alpha}}
 {1}/{|\Gamma_f|} }{\sum_{f \in F} 1/ {|\Gamma_f|}},$$
 where $F_\alpha= \{ f \in F\, :\, B^G_{\rho(\alpha)}(f) \in \alpha \}.$
\end{Theorem}

\begin{proof}
Fix an arbitrary $r \in \NN$, $\delta > 0$, and a fundamental domain
of vertices $F \subset V$ of finite co-volume.
The corresponding covering map is denoted by $\pi:V \to F$.  We construct a finite graph $G_r$ such that
 \begin{gather}\left| p(G_r,\alpha) - \big( \sum_{F_{\alpha}} 1/|\Gamma_f| \big) / \big( \sum_F 1/|\Gamma_f| \big) \right|
 < g(\delta) \tag{$\Diamond$} \label{ineq: approximation}
  \end{gather}
 holds for any $\alpha \in \mathcal{A}_r^D$, where $g: [0,1) \to \RR_{\geq 0}$ is a function solely dependent on $\delta$ with $\lim_{\delta \to 0} g(\delta) = 0$.  \\
 We first find some finite $\widetilde{F} \subseteq F$ such that
 \begin{eqnarray*}
 \sum_{f \in \widetilde{F}} \frac{1}{|\Gamma_f|} \leq \sum_{f \in F} \frac{1}{|\Gamma_f|}  \leq
 \sum_{f \in \widetilde{F}} \frac{1}{|\Gamma_f|} + \delta.
 \end{eqnarray*}
 With this notion at hand, we set
 \[
 \widetilde{B}^G_r(\widetilde{F}) := \big\{ x \in B^G_r(\widetilde{F})\,|\, \pi(x) \in \widetilde{F}  \big\} = B^G_r(\widetilde F) \cap \pi^{-1}(\widetilde{F}),
 \]
 where $B^G_r(\widetilde{F})$ denotes the union of $r$-balls in $G$ with centers in $\widetilde{F}$.
 Since the action of $\Gamma$ has finite stabilizers, for each $x \in V$ there exist finitely many elements $\gamma \in \Gamma$ such that $x = \gamma \pi(x)$ holds.
 We let
 $$T:= T_{\widetilde{F}} := \{\gamma \in \Gamma \mid     x = \gamma\, \pi(x) \mbox{ for some } x \in B^G_r(\widetilde{F}) \text{ with } \pi(x) \in \widetilde F\}$$
denote the collection of the $\gamma$ corresponding to elements in the $r$-Ball around $\widetilde{F}$. Note that due to the finiteness of
$\widetilde{F}$, as well as the finiteness of the stabilizers, the set $T$ is finite.
Note further that for all $f \in \ow{F}$, we have $\Gamma_f \subseteq T$. 
We set
\[
\ow{T} := \big( T \cup T^{-1} \big)^4 = \{fghl \mid f,g,h,l \in T \cup T^{-1}\}.
\]
Note that $T \cup T^{-1} \subseteq \ow{T}$ since $e \in T$.
We now use the fact that $\Gamma$ is sofic in order to find almost-homomorphisms with desirable properties. Precisely, we choose $N \in \NN$ and $\sigma: \ow{T} \to \operatorname{Sym}(N)$ such that there is some
set $S \subseteq \{1,\ldots,N\}$ with $|S| \geq (1-\delta / |\widetilde{F}|)N$ and for each $i \in S$, we get
\begin{enumerate}[(a)]
\item $\sigma_{\gamma\gamma^{\prime}}(i) = \sigma_{\gamma} \sigma_{\gamma^{\prime}}(i)$ for all $\gamma, \gamma^{\prime} \in (T \cup T^{-1})^2$,
\item $\sigma_{\gamma \gamma^{\prime} \gamma^{\prime\prime}}(i) = \sigma_{\gamma} \sigma_{\gamma^{\prime}} \sigma_{\gamma^{\prime\prime}}(i)$
for all $\gamma, \gamma^{\prime}, \gamma^{\prime\prime} \in T \cup T^{-1}$,
\item $\sigma_{\theta_1}(i) = \sigma_{\theta_2}(i)$ if and only if $\theta_1 = \theta_2$ for all $\theta_1, \theta_2 \in \ow{T}$,
\item $\sigma_{e}(i) = i$. 
\end{enumerate}
We remark that the existence of such an $N \in \NN$, and such a 
map $\sigma: \ow{T} \to \operatorname{Sym}(N)$ is a straight forward 
consequence from the definition of soficity by modifying suitable 
almost homomorphisms in the sense of Definition~\ref{defi:soficgroups}, for some $\varepsilon > 0$ chosen small enough in terms of $\delta$ and
in terms of the sizes of $\ow{T}$ and $\widetilde{F}$.  By adjusting parameters again if necessary, we can additionally find some set $\tilde{S} \subseteq S$
with
\begin{gather} \label{eqn:sizeS}
|\tilde{S}| \geq (1-\delta/|\widetilde{F}|)|S| \geq (1-2\delta / |\widetilde{F}|)N \tag{$\bigcirc$}
\end{gather}
such that
\begin{enumerate}[(e)]
\item $\sigma_{\gamma}(i) \in S$ for all $i \in \tilde{S}$ and 
each $\gamma \in T \cup T^{-1}$.
\end{enumerate}

We will now use this map $\sigma$ to construct the graph $G_r$. To this end, we start with an auxiliary construction. For $f\in \widetilde F$ and $i,j \in S$ we say that $i$ and $j$ are {\em $f$-pre-connected}, if there exists $\gamma \in \Gamma_f$ such that $\sigma_{\gamma}(i) = j$. In this case, we write $i \sim_f j$.  

\medskip

{\em Claim~1:} $\sim_f$ is an equivalence relation on $S$. Moreover, for any $i \in \widetilde S$ the corresponding equivalence class $[i]_f$ satisfies
$$|[i]_f|  = |\{j \in S \mid i \sim_f j\}| =  |\Gamma_f|.$$

\medskip

{\em Proof of the claim.} We have $e \in \Gamma_f$. Hence it follows from (d) that $i \sim_f i$ and so $\sim_f$ is reflexive. Next we show its symmetry. If $i \sim_f j$, then by definition there exists $\gamma \in \Gamma_f$ such that $\sigma_\gamma (i) = j$. Since $i,j \in S$, properties  (a) and (d) yield $\sigma_{\gamma^{-1}}(j) = i$ proving $j \sim_f i$. With the same arguments it follows that $i \sim_f j$ and $j \sim_f k$ implies $i \sim_f k$ and so $\sim_f$ is transitive.  By definition of $\widetilde{S}$ for every $\gamma \in \Gamma_f$ and $i \in \widetilde{S}$ we have $\sigma_\gamma(i) \in S$ and therefore $i \sim_f \sigma_\gamma(i)$. Moreover, if $\sigma_\gamma(i) = \sigma_{\gamma'}(i)$, then (c) implies $\gamma' = \gamma$. This proves the desired equality for $|[i]_f|$ and the claim follows.

\medskip

This puts us in the position to define the graph $G_r$. We let $S_f$ the set of equivalence classes of $\sim_f$ and $\widetilde S_f :=\{[i]_f \in S_f \mid i \in \widetilde{S}\}$ the set of equivalence classes which have one representative in $\widetilde S$. Since the equivalence relation induces a partition of $S$, the previous claim on the size of $[i]_f$ for $i \in \widetilde{S}$ yields
\begin{align*}
\frac{|\widetilde{S}|}{|\Gamma_f|} \leq |\widetilde S_f | \leq |S_f| \leq \frac{|S|}{|\Gamma_f|} + |S \setminus \widetilde{S}|. \tag{$\bigtriangleup$} \label{ineq: preconnectedness1}
\end{align*}
%
%
%
%
%
%
%
%
%
We define the vertex set of $G_r$ as  $V_r := \big\{ (f,[i]_f)\,|\, f \in \ow{F}, [i]_f \in {S}_f \big\}$
and we connect two vertices $(f,[i]_f), (g,[j]_g)$ by an edge if there exist $\gamma \in T \Gamma_f$ and $\gamma' \in T\Gamma_g$   such that    $\gamma f \sim \gamma' g$ in $G$ and  $\sigma_{\gamma}(i) = \sigma_{\gamma'}(j)$ holds. This definition is independent of the particular choice of $i,j$. Indeed, if $i' \in [i]_f$ and $j' \in [j]_g$, then there exist $\gamma_i \in \Gamma_f$ and $\gamma_j \in \Gamma_g$ with $\sigma_{\gamma_i}(i') = i$ and $\sigma_{\gamma_j} (j') = j$. Since stabilizers are subgroups we have $\gamma\gamma_i \in T \Gamma_f$ and $\gamma'\gamma_j \in T \Gamma_g$. From (a) we infer
$$\sigma_{\gamma\gamma_i}(i') = \sigma_\gamma(i) = \sigma_{\gamma'}(j) = \sigma_{\gamma^{\prime} \gamma_j}(j').$$
This implies independence of chosen representatives.

To show the desired statement, for each $i \in \ow{S}$, we introduce the maps
$$\varphi_i:\ow{B}_r^G(\ow{F}) \to V_r,\, x \mapsto \big(\pi(x), [\sigma_{\gamma^{-1}}(i)]_{\pi(x)}\big),$$
where $\gamma$ solves $x = \gamma \, \pi(x)$. Note that since $i \in \widetilde S$, we have $\sigma_{\gamma^{-1}}(i) \in S$ so that $[\sigma_{\gamma^{-1}}(i)]_{\pi(x)}$ exists. Naturally, we have to show that this definition does not depend on the choice of $\gamma$.  If $\gamma$ and $\gamma^{\prime}$ are two solutions, we necessarily
get $\gamma^{-1}\gamma^{\prime} \in \Gamma_{\pi(x)}$. Since $i \in \ow{S}$, property~(e) combined with the properties (a) and (d) yields
\[
\sigma_{\gamma^{-1}\gamma^{\prime}} \sigma_{\gamma^{\prime\, -1}} (i) 
\stackrel{\mathrm{(a)}}{=} 
\sigma_{\gamma^{-1}} \sigma_{\gamma^{\prime}} \sigma_{\gamma^{\prime\, -1}}(i)
\stackrel{\mathrm{(a)}}{=}
\sigma_{\gamma^{-1}}\sigma_{e}(i) 
\stackrel{\mathrm{(d)}}{=} \sigma_{\gamma^{-1}}(i).
\]
Hence, $\sigma_{\gamma'^{-1}}(i) \sim_{\pi(x)} \sigma_{\gamma^{-1}}(i)$.

\medskip

We need to prove that the mappings $\varphi_i$ preserve the local graph structures. 


\medskip

{\em Claim 2:} For all $i \in \ow{S}$  the map $\varphi_i$ is a graph isomorphism onto its image. 

\medskip

{\em Proof of the claim:} Let $i \in \ow{S}$. We show that $\varphi_i$ is injective. For $x,y$ we choose $\gamma_x,\gamma_y \in T $ with $x = \gamma_x \pi(x)$ and $y = \gamma_y \pi(y)$. If  $\varphi(x) = \varphi(y)$, then $\pi(x) = \pi(y)$ and
$$[\sigma_{\gamma_x^{-1}}(i)]_{\pi(x)} =  [\sigma_{\gamma_y^{-1}}(i)]_{\pi(y)}  = [\sigma_{\gamma_y^{-1}}(i)]_{\pi(x)}.$$
Hence, there exists $\gamma \in \Gamma_{\pi(x)}$ with $\sigma_\gamma \sigma_{\gamma_x^{-1}}(i) = \sigma_{\gamma_y^{-1}}(i)$. By (a) and (c) we obtain $\gamma \gamma_x^{-1} = \gamma_y^{-1}$. This implies 
$$y = \gamma_y \pi(y) = \gamma_y \pi(x) = \gamma_x \gamma^{-1} \pi(x) = x$$
and the injectivity is proven.


Next we prove that $\varphi_i$ and its inverse preserve edge relations. Let $x,y \in \ow{B}_r(\ow{F})$ with $(\pi(x),[\sigma_{\gamma_x^{-1}}(i)]_{\pi(x)}) \sim (\pi(y),[\sigma_{\gamma_y^{-1}}(i)]_{\pi(y)})$.  This is equivalent to the existence of $\gamma \in T \Gamma_{\pi(x)}$ and $\gamma' \in T \Gamma_{\pi(y)}$ with $\gamma \pi(x) \sim \gamma' \pi(y)$ in $G$ and $\sigma_\gamma \sigma_{\gamma_x^{-1}}(i) = \sigma_{\gamma^{\prime}} \sigma_{\gamma_y^{-1}}(i)$. According to (a) and (c) this in turn is equivalent to the existence of $\gamma \in T \Gamma_{\pi(x)}$ and $\gamma' \in T \Gamma_{\pi(y)}$ with $\gamma \pi(x) \sim \gamma' \pi(y)$ in $G$  and $\gamma \gamma_x^{-1} = \gamma' \gamma_y^{-1}$. Since $\gamma \gamma_x^{-1}  x = \gamma \pi(x)$, $\gamma' \gamma_y^{-1}  y = \gamma' \pi(y)$ and $\Gamma$ acts on $G$ by graph isomorphisms, the latter statement is equivalent to $x \sim y$ in $G$ and the isomorphism properties for $\varphi_i$ are proven.

For $f \in \widetilde{F}$ and $i \in \widetilde S$ we have $\varphi_i(f) = (f, [i]_f)$.

\medskip

To finish the proof of the theorem fix $\alpha \in \mathcal{A}_r^D$.
Since $\mathcal{A}_r^D$ is finite, by increasing the set $\ow{F}$ if necessary (at the beginning of the proof), we can further assume that for the set
\[
\ow{F}_{\alpha} := \big\{ f \in F_{\alpha} \cap \ow{F}\,|\, \pi(x) \in \ow{F} \mbox{ for all } x \in B_r(f)  \big\},
\]
we have
$\sum_{f \in F_{\alpha}} 1/|\Gamma_f| \leq \sum_{f \in \ow{F}_{\alpha}} 1/|\Gamma_f| + \delta$.
Combining the fact that the $\varphi_i$ are graph isomorphisms, which map $f$ to $(f,[i]_f)$, with the bounds  in~\eqref{ineq: preconnectedness1}, we have
\begin{eqnarray*}
\sum_{f \in \ow{F}_{\alpha}}  \frac{|\ow{S}|}{|\Gamma_f|} \leq \big| \big\{ (f,[i]_f) \in V_r\,|\, B_{\rho(\alpha)}^{G_r}\big( (f,[i]_f) \big) \in \alpha \big\} \big|
 \leq  \sum_{f \in \ow{F}_{\alpha}}\frac{|S|}{|\Gamma_f|} + |\widetilde{F}| \cdot |S \setminus \widetilde{S}|.
\end{eqnarray*}
Recall that $|S|\leq N$ and  by inequality~\eqref{eqn:sizeS}, we have
$$|\ow{S}| \geq (1 - \delta/|\widetilde{F}|) |S| \geq (1- 2\delta/|\widetilde{F}|) N.$$
Consequently,
\begin{eqnarray*}
\sum_{f \in \ow{F}_{\alpha}} \frac{(1-2\delta)N}{|\Gamma_f|} &\leq& p\big(G_r,\alpha \big)\,|V_r| \quad \leq \quad \sum_{f \in \ow{F}_{\alpha}} 
\frac{N}{|\Gamma_f|} + \delta \cdot N , \\
\sum_{f \in \ow{F}} \frac{(1-2\delta)N}{|\Gamma_f|} &\leq& |V_r| \quad \leq \quad \sum_{f \in \ow{F}} \frac{N}{|\Gamma_f|} + \delta \cdot N.
\end{eqnarray*}
These inequalities and the fact that the sums over $\ow{F}$ and over $\ow{F}_{\alpha}$ are close to the sums over $F$, respectively
over $F_{\alpha}$ imply
\begin{align*}
(1-2\delta) \frac{\sum_{f \in F_\alpha} 1/|\Gamma_f|\, -\, \delta}{\sum_{f \in F } 1/|\Gamma_f| \,+\, \delta} \leq p(G_r,\alpha) \leq \frac{1}{1-2\delta} \frac{\sum_{f \in F_\alpha} 1/|\Gamma_f| \, + \, \delta}{\sum_{f \in F } 1/|\Gamma_f| - \delta}.
\end{align*}
With this at hand, routine calculations show that for small $\delta$ the resulting graph $G_r$ satisfies inequality~\eqref{ineq: approximation} with
$g(\delta) =  C \delta$ for a large enough constant $C$ depending on $\sum_{f \in F } 1/|\Gamma_f|$ and $\sum_{f \in F_{\alpha}} 1/|\Gamma_f|$.  This finishes the proof.
\end{proof}

The previous theorem can be reformulated in terms of weak convergence of measure graphs. It states that for any periodic graph $(G,\Gamma)$ as above there exists a sequence of finite graphs $(G_n)$ such that the associated (normalized) measure graphs $((G_n,M_n))$ weakly converge to the measure graph $\Big(G, \big(\sum_{f} 1/|\Gamma_f|\Big)^{-1} M \big)$. With this observation we immediately obtain the following from Theorem~\ref{theorem:convergence of zeta functions}.

\begin{Corollary} \label{cor:soficapprox}
Let $(G,\Gamma)$ be a periodic graph with vertex degree bounded by $D$.
Assume further that the group $\Gamma$ is sofic and that $F$ is a fundamental domain of finite co-volume $\mathfrak{v} := \sum_{f \in F} 1/|\Gamma_f|$. Then, there is a weakly convergent sequence of finite graphs $(G_n)$ such that
\[
\lim_{n \to \infty} Z_{G_n, {\rm norm}}^{\mathfrak{v}} = Z_{(G,\Gamma)},
\]
uniformly on compact subsets of $B_{(D-1)^{-1}}:= \{u \in \CC\, : \, |u| < (D-1)^{-1}\}$. 
\end{Corollary}

\begin{Remark}
(a) The above corollary is an extension of the approximating
theorems of~\cite{CMS} dealing with residually finite groups acting
freely on a regular graph  and~\cite{GZ} dealing with limits of
covering sequences of finite, regular graphs.  Further, it is the
natural extension of the approximation result in \cite{GIL08} which
is concerned with amenable graphs.
Additional generality is provided by fact that in the framework of a finite measure, we allow for infinite fundamental domains, as well as for proper (but not necessarily free) actions of $\Gamma$ on $V$. The crucial idea in the proof for overcoming the lack of freeness  is to identify vertices which are linked by almost homomorphisms $\sigma_{\gamma}$ arising from stabilizer elements $\gamma$ (``pre-connectedness'').



(b) For graphs with positive Cheeger constant, finite exhaustions
might converge, but in general, the geometry of the resulting limit
will be completely different from the one of the original objects.
One example, as e.g.\@ considered in the introduction of the
paper~\cite{BLS}, is a sequence of finite, regular trees which converge
towards the graphing consisting almost surely of
infinite Canopy trees. Those are by no means infinite regular
trees.
Thus, the approximation {\em through
induced subgraphs} fails in general in non-amenable situations.
However, sofic approximations can still be found in many situations.


(c) The  weakly converging sequence of finite graphs constructed in
Theorem \ref{thm:soficapprox} for periodic graphs with a sofic group
action  can also be used for spectral approximation of suitable
self-adjoint operators. In fact, whenever a sequence of finite
graphs converges weakly then weak convergence of the normalized
empirical spectral distributions of corresponding operators to a
limit follows and this limit can be expressed through a trace on the
von Neumann algebra associated with the limit graphing (see
\cite{Ele08, Ele08a, Pog} for detailed explanations). In this way,
Theorem \ref{thm:soficapprox} could be used to recover parts of the
results of \cite{Sch2}. (The results of \cite{Sch2} are more general
in that they  allow for unbounded operators and include some
randomness.) For hyperfinite graphs even uniform convergence of the
normalized empirical spectral distributions can be shown, cf.\@
\cite{Ele08, Pog, Pog14}.

(d) Note that the previous corollary is not a special case of
Theorem~\ref{theorem:weakly convergent graphs}. Indeed, we identify
the limit Zeta function as the Zeta function of a periodic graph
instead of that of a graphing. 
This  is possible due to our concept of convergence 
of measure graphs. This allows us to directly describe the periodic
graph as the limit object instead of giving rise to a graphing with
the same local statistics as the original graph.
\end{Remark}

\end{document}